\documentclass[11pt,twoside]{article}
\usepackage{amsfonts}
\usepackage{amsmath}
\usepackage{amsthm}
\usepackage{graphicx}

\usepackage{enumerate}

\usepackage[urlcolor=blue]{hyperref}

\def\qed{
\vbox{\hrule\hbox{\vrule\hbox to 5pt{\vbox to 8pt{\vfil}\hfil}\vrule}\hrule}}
\def\qedsmall{
\vbox{\hrule\hbox{\vrule\hbox to 4pt{\vbox to 4pt{\vfil}\hfil}\vrule}\hrule}}

\def\endproof{\unskip \nobreak \hskip0pt plus 1fill \qquad \qed}

\def\9{\"u9}
\newtheoremstyle{custom}{3pt}{3pt}{}{}{\bfseries}{:}{.5em}{}
\theoremstyle{custom}
\newtheorem{theorem}{Theorem}[section]

\newtheorem{definition}[theorem]{Definition}
\newtheorem{remark}[theorem]{Remark}

\newtheorem{proposition}[theorem]{Proposition}

\numberwithin{equation}{section}
\numberwithin{figure}{section}
\numberwithin{table}{section}
\newenvironment{keywords}{{\bf Keywords:}}{}
\renewenvironment{abstract}{\begin{small} {\bf Abstract:}}{\end{small}}

\def\R{\mathbb{R}}

\def\N{\mathbb{N}}

\def\KK{{\cal K}}

\def\WW{{\cal W}}

\def\OO{\mathcal{O}}

\def\eps{\varepsilon}

\def\beq{\begin{equation}}
\def\eeq{\end{equation}}
\def\bea{\begin{eqnarray}}
\def\eea{\end{eqnarray}}
\def\beaa{\begin{eqnarray*}}
\def\eeaa{\end{eqnarray*}}

\usepackage{color}

\begin{document}
\bibliographystyle{siam_nolines}

\title{Computing Lyapunov functions using deep neural networks}
\author{
Lars Gr\"une\\
Mathematical Institute\\
University of Bayreuth\\
95440 Bayreuth, Germany\\
{\tt lars.gruene@uni-bayreuth.de}
}
\date{\today}

\maketitle
\renewcommand{\thefootnote}{\arabic{footnote}}
\pagestyle{myheadings}
\thispagestyle{plain}
\markboth{LARS GR\"UNE}{COMPUTING LYAPUNOV FUNCTIONS USING DEEP NEURAL NETWORKS}

\begin{abstract} 
We propose a deep neural network architecture and a training algorithm for computing approximate Lyapunov functions of systems of nonlinear ordinary differential equations. Under the assumption that the system admits a compositional Lyapunov function, we prove that the number of neurons needed for an approximation of a Lyapunov function with fixed accuracy grows only polynomially in the state dimension, i.e., the proposed approach is able to overcome the curse of dimensionality. We show that nonlinear systems satisfying a small-gain condition admit compositional Lyapunov functions. Numerical examples in up to ten space dimensions illustrate the performance of the training scheme.
\end{abstract}



\section{Introduction}
Lyapunov functions are one of the key tools for the stability analysis of nonlinear systems. They do not only serve as a certificate for asymptotic stability of an equilibrium but also allow to give estimates about its domain of attraction or to quantify its robustness with respect to perturbations, for instance, in the sense of input-to-state stability. However, explicit analytic expressions for Lyapunov functions are often not available. Hence, the numerical computation of Lyapunov functions has attracted significant attention during the last decades. Known approaches use series expansions \cite{KiNB82}, finite element approaches \cite{CaGW00}, representations by radial basis functions \cite{Gies07} or piecewise affine functions, see \cite{Hafs07}, or sum-of-squares techniques, see \cite{AndP15} and the references therein. For a comprehensive overview we refer to the survey by \cite{GieH15}. Often, a characterization of the Lyapunov function via a suitable partial differential equation (PDE) such as Zubov's equation \cite{Zubo64} is used as the basis for these numerical computations.

The usual approaches have in common that the number of degrees of freedom needed for storing the Lyapunov function (or an approximation thereof with a fixed approximation error) grows very rapidly --- typically exponentially --- with the dimension of the state space. This is the well known curse of dimensionality, which leads to the fact that the mentioned approaches are confined to low dimensional systems.

In general, the same is true if a deep neural network is used as an approximation architecture. While it is known that such a network can approximate every $C^1$-function arbitrarily well, see \cite{Cybe89,HoSW89}, the number of neurons needed for this purpose typically grows exponentially with the state dimension, as well, see \cite[Theorem 2.1]{Mhas96} or Theorem \ref{thm:univ}, below. However, this situation changes 
if additional structural assumptions are imposed, which is the approach we follow in this paper. Recently, there has been a large activity in exploiting suitable structural properties for solving high-dimensional PDEs using neural networks \cite{DaLM20,EHJ17,BeGJ20,HaJE18,HuPW20,HJKN20,HuJT19,ReiZ19,SirS18} and since Lyapunov functions can also be represented by PDEs, these results provided the first source of inspiration for this paper. 

As we will show in this paper, in the Lyapunov function context a suitable property for making the neural network approach efficient is the existence of what we call a compositional Lyapunov function, cf.\ Definition \ref{def:complf}, below. The importance of compositionality for overcoming the curse of dimensionality is explained in \cite{PMRML17}, and this reference provides the second source of inspiration for this paper. Using similar arguments as in \cite{PMRML17}, we show that a suitably designed deep neural network can compute approximations of compositional Lyapunov function with a given required accuracy using a number of neurons that grows only polynomially with the dimension of the system. In other words, we show that the curse of dimensionality can be avoided. 

The important question then is how restrictive the assumption of the existence of a compositional Lyapunov function is. It turns out that a classical systems theoretic tool for stability analysis of large-scale systems, namely small-gain analysis --- here in its nonlinear form based on input-to-state stability, see, e.g., \cite{DaIW11,DaRW10,JiMW96,JiTP94,Ruef07} --- provides conditions on the dynamics under which a compositional Lyapunov function exists. This insight together with the design of a corresponding deep neural network architecture with two hidden layers constitutes the theoretical contribution of this paper. This is complemented by an algorithmic contribution in form of a loss function for a training algorithm for neural networks that is based on a suitable partial differential inequality, and by numerical tests that illustrate the efficiency of the proposed ``DeepLyapunov'' method.

There have been earlier attempts to use neural networks for approximating Lyapunov functions. The paper \cite{Serp05} proposes a learning algorithm based on increments instead of derivatives, which relies on successive updates of the network parameters rather than a standard learning algorithm. This paper does not provide numerical examples illustrating the performance of the approach. In \cite{PetP06} only a local Lyapunov function is computed, by using local derivative information in the learning algorithm. In the paper \cite{NKSJ08} the coefficients of a polynomial Lyapunov functions are computed, rather than representing the Lyapunov function directly by a neural network as in our paper. There are also various papers dealing with the more general problem of computing control Lyapunov functions (clfs). \cite{LonB93} considers this problem by assuming exact representability of the Lyapunov function by a neural network with one hidden layer. The paper \cite{RiBK18} considers clfs of a particular quadratic form in discrete time. It implements the decrease condition via classification rather than differential inequalities. The paper \cite{KhZB14} considers clfs for models from robotics and optimizes the parameters of a quadratic Lyapunov function candidate. Finally, among the many papers considering neural network based solutions of Hamilton-Jacobi-Bellman equations some also yield Lyapunov functions. For instance, this is done in \cite{AbKL05}, in which neural networks with one hidden layer are considered. A related technique is to represent stabilizing controllers by neural networks, where again Lyapunov functions can be used for the stability analysis, see, e.g., \cite{LeJY98}, or structural properties allowing for such a representation are investigated \cite{Sont92}. The latter reference is conceptually similar to this paper in the sense that beneficial structural properties of the dynamics and the neural networks are investigated, however, the problem under consideration is different. In summary, all these references differ in several aspects from the approach proposed in this paper. Yet, the main difference is that none of them carries out a complexity analysis or provides a network structure that is provably able to overcome the curse of dimensionality and performs well for higher dimensional nonlinear systems in numerical experiments. This is the distinctive contribution of this paper.

The remainder of the paper is organized as follows. In Section \ref{sec:problem} we formulate the problem. In Section \ref{sec:complf} we explain the concept of compositional Lyapunov functions and its relation to small-gain theory. Section \ref{sec:nn} gives a brief introduction into neural networks, mainly in order to clarify the notation used in Section \ref{sec:nnlf}. In this section we propose a neural network architecture and prove that it allows to store approximations to compositional Lyapunov functions avoiding the curse of dimensionality. In Section \ref{sec:training} we propose loss functions for a training algorithm that allows to actually compute Lyapunov functions using the proposed neural networks. Numerical results illustrating the performance of our approach are given in Section \ref{sec:numerics}. In Section \ref{sec:discussion} we discuss various aspects and extensions of our approach before we conclude our paper in Section \ref{sec:conclusion}. The results from Section \ref{sec:nnlf} are contained in preliminary form in the conference paper \cite{Grue20}. However, \cite{Grue20} did not address training algorithms nor did it present numerical results. Moreover, the proofs in Section \ref{sec:nnlf} are given in more detailed form in the present paper.

\section{Problem Formulation}\label{sec:problem}

We consider nonlinear ordinary differential equations of the form
\begin{equation} \dot x(t) = f(x(t)) \label{eq:sys}\end{equation}
with a Lipschitz continuous $f:\R^n\to\R^n$. We assume that $x=0$ is an asymptotically stable equilibrium and that $K_n\subset\R^n$ is a compact set in its domain of attraction.

It is well known (see, e.g., \cite{Hahn67}) that asymptotic stability is equivalent to the existence of a Lyapunov function according to the following definition.
\begin{definition} A continuously differentiable function $V:O\to\R$ defined on an open set $O$ with $0\in O$ is a Lyapunov function if it satisfies the following properties: $V(0)=0$, $V(x)>0$ for all $x\ne 0$, and the orbital derivative $DV(x)f(x)$, i.e., the derivative $DV$ of $V$ multiplied with the direction of the vector field $f$, satisfies
\begin{equation} DV(x)f(x) \le - h(x) \label{eq:lf}\end{equation}
for a function $h:\R^n\to\R$ with $h(x)>0$ for all $x\in O\setminus\{0\}$. If $O=\R^n$ and $V(x)\to\infty$ as $\|x\|\to\infty$, then $V$ is called a global Lyapunov function.
\end{definition}

If $V$ is a Lyapunov function, then any connected component of a sublevel set of $V$ containing $0$ is part of the domain of attraction of $x=0$. Hence, $K_n$ is in the domain of attraction of $x=0$ if it is contained in such a set. In this case we call $V$ a Lyapunov function on $K_n$.

It is our goal in this paper to design a neural network that is able to compute an approximation of such a Lyapunov function on $K_n\subset\R^n$ in an efficient manner. Efficient here is meant in the sense that the computational effort as well as the storage effort grow moderately with the space dimension. While this will not be possible in general, we will show that it works for Lyapunov functions satisfying a particular structure, which we call compositional Lyapunov functions. This structure is motivated by recent results on approximation properties of neural networks \cite{PMRML17}, but it turns out that it is also well known in systems theory, as it corresponds to a particular kind of a small-gain condition, which we describe in the next section. Throughout this paper, we will consider families of set $K_n\subset\R^n$ in varying space dimensions $n\in\N$ for which we make the standing assumption
\beq \mbox{there exists $C>0$ with $K_n\subset [-C,C]^n$ for all $n\in\N$}\label{eq:Kn}\eeq
in order to avoid that the sets $K_n$ grow unboundedly in the $\ell_\infty$-norm with the dimension $n$.

\section{Compositional Lyapunov functions and small-gain theory}\label{sec:complf}

The particular compositional structure we consider is motivated by \cite{PMRML17}, where the approximation of general functions via neural networks is considered. In order to define this structure, the system \eqref{eq:sys} is divided into $s$ subsystems $\Sigma_i$ of dimensions $d_i$, $i=1,\ldots,s$. To this end, the state vector $x=(x_1,\ldots,x_n)^T$ and the vector field $f$ are split up as
\[ x = \left(\begin{array}{c} z_1\\ z_2\\ \vdots \\ z_s\end{array}\right) \; \mbox{ and } \; f(x) = \left(\begin{array}{c} f_1(x)\\ f_2(x)\\ \vdots \\ f_s(x)\end{array}\right),\]
with $z_i=(x_{\hat d_{i-1}+1},\ldots,x_{\hat d_i})\in\R^{d_i}$ and $f_i:\R^n\to\R^{d_i}$ denoting the state and dynamics of each $\Sigma_i$, $i=1,\ldots,s$, with state dimension $d_i\in\N$ and $\hat d_i=\sum_{j=1}^i d_j$. With 
\[ z_{-i} := \left(\begin{array}{c} z_1\\ \vdots \\ z_{i-1}\\z_{i+1}\\ \vdots \\ z_s\end{array}\right) \]
and by rearranging the arguments of the $f_i$, the dynamics of each $\Sigma_i$ can then be written as
\[ \dot z_i(t) = f_i(z_i(t),z_{-i}(t)), \quad i=1,\ldots,s. \]
Using this decomposition, we can define the following Lyapunov function structure\footnote{In order to avoid an overly technical presentation, the exposition in this section is limited to global Lyapunov functions.}.
\begin{definition} A Lyapunov function $V$ for \eqref{eq:sys} is called {\em compositional}, if there exist $C^1$-functions $\widehat V_i:\R^{d_i}\to\R$ such that $V$ is of the form
\beq V(x) = \sum_{i=1}^s \widehat V_i(z_i). \label{eq:sglf}\eeq
\label{def:complf}\end{definition}

In the remainder of this section we discuss conditions on $f$ under which a Lyapunov function of the form \eqref{eq:sglf} exists.

One situation in which a system \eqref{eq:sys} admits a compositional Lyapunov function is when the $f_i$ do not depend on $z_{-i}$, i.e., if $f_i(z_i,z_{-i}) = f_i(z_i)$. This means that the subsystems are completely decoupled. In this case, consider Lyapunov functions $\hat V_i$ of $\dot x_i = f_i(x_i)$ on compact sets $\widehat K_i\subset\R^{d_i}$, and $V$ from \eqref{eq:sglf}. Then, clearly $V(x)\ge 0$ and $V(x)=0$ if and only if $x=0$. Moreover,
\[ DV(x)f(x) = \sum_{i=1}^s DV_i(z_i)f_i(z_i) <0 \]
for all $x\in K_n=\prod_{i=1}^s \widehat K_i$ with $x\ne 0$. 

Assuming that $f$ decomposes into $s$ completely decoupled subsystems is relatively restrictive. Fortunately, a similar construction can also be made if the $f$ are coupled, provided the coupling is such that it does not affect the stability of the overall subsystem. The systems theoretic tool for this approach is nonlinear small-gain theory, which relies on the input-to-state stability (ISS) property introduced in \cite{Sont89}. It goes back to \cite{JiMW96,JiTP94} and in the form for large-scale systems we require here it was developed in the thesis \cite{Ruef07} and in a series of papers around 2010, see, e.g., \cite{DaIW11,DaRW10} and the references therein. ISS small-gain conditions can be based on trajectories or Lyapunov functions and exist in different variants. Here, we use the variant that is most convenient for obtaining approximation results because it yields a smooth Lyapunov function. We briefly discuss one other variant in Section \ref{sec:discussion}\eqref{it:maxlf}.

For formulating the small gain condition, we assume that for the subsystems $\Sigma_i$ there exist $C^1$ ISS-Lyapunov functions $V_i:\R^{d_i}\to\R$, satisfying for all $z_i\in\R^{d_i}$ $z_{-i}\in\R^{n-d_i}$
\[ DV_i(z_i)f_i(z_i,z_{-i}) \le -\alpha_{i}(\|z_i\|) + \sum_{j\ne i} \gamma_{ij}(V_j(z_j)) \]
with rates $\alpha_i\in\KK_\infty$ and gains $\gamma_{ij}\in\KK_\infty$,\footnote{As usual, we define $\KK_\infty$ to be the space of continuous functions $\alpha:[0,\infty)\to[0,\infty)$ with $\alpha(0)=0$ and $\alpha$ is strictly increasing to $\infty$.} $i,j=1\ldots,s$, $i\ne j$. 
Here, the states $z_{-i}$ of the other subsystems are interpreted as the input to the $i$-th subsystem and the term $\sum_{j\ne i} \gamma_{ij}(V_j(z_j))$ in the ISS inequality quantifies how much this input affects the stability of the $i$-th subsystem. Particularly, the larger the ISS-gains $\gamma_{ij}$ are, the more the other systems' influence can affect the decrease of the Lyapunov function $V_i$. Setting $\gamma_{ii}:= 0$, we define the map $\Gamma:[0,\infty)^s\to[0,\infty)^s$ by 
\[ \Gamma(r) := \left(\sum_{j=1}^s\gamma_{1j}(r_j),\ldots,\sum_{j=1}^s\gamma_{sj}(r_j)\right)^T\]
and the diagonal operator $A:[0,\infty)^s\to[0,\infty)^s$ by 
\[ A(r) := \left(\alpha_{1}(r_1),\ldots,\alpha_{s}(r_s)\right)^T.\]

\begin{definition} \label{def:sg} We say that \eqref{eq:sys} satisfies the small-gain condition, if there is a decomposition into subsystems $\Sigma_i$, $i=1,\ldots,s$, with ISS Lyapunov functions $V_i$ satisfying the following condition: there are bounded positive definite\footnote{A continuous function $\rho:[0,\infty)\to[0,\infty)$ is called positive definite if $\rho(0)=0$ and $\rho(r)>0$ for all $r>0$.} functions $\eta_i$, $i=1,\ldots,s$, satisfying $\int_0^\infty \eta_i(\alpha_i(r))dr = \infty$ and such that for $\eta=(\eta_1,\ldots,\eta_s)^T$ the inequality
\[ \eta(r)^T\Gamma\circ A(r) < \eta(r)^Tr\]
holds for all $r\in[0,\infty)^s$ with $r\ne 0$.
\end{definition}

It is easily seen that this inequality is satisfied whenever the gains $\gamma_{ij}$ are sufficiently small, which explains the name small-gain condition. 
The following theorem then follows from Theorem 4.1 in \cite{DaIW11}.

\begin{theorem}
Assume that the small-gain conditions from Definition \ref{def:sg} hold. Then $V$ from \eqref{eq:sglf} is a Lyapunov function for the $C^1$-functions $\widehat V_i:\R^{d_i}\to\R$ given by
\[ \widehat V_i(z_i):= \int_0^{V_i(z_i)} \lambda_i(\tau)d\tau \]
where $\lambda_i(\tau) = \eta_i(\alpha_i(\tau))$. 
\label{thm:sg}\end{theorem}

In \cite{DaIW11}, the property from Definition~\ref{def:sg} is called a \emph{weak} small-gain condition. This is because if the system \eqref{eq:sys} has an additional input (that is taken into account in the assumptions on the $V_i$), then the construction of $V$ yields an integral ISS Lyapunov function as opposed to an ISS Lyapunov function. Under a stronger version of the small-gain condition, the same construction yields an ISS Lyapunov function. We briefly discuss corresponding extensions of our approach in Section \ref{sec:discussion}\eqref{it:u}.

We note that for various reasons small-gain conditions are not easy to check and to apply: the gains $\gamma_{ij}$ may be difficult to estimate, the functions $\eta_i$ may be hard to find and, above all, appropriate Lyapunov functions $V_i$ for the subsystems may be nontrivial to construct. However, none of these ingredients need to be known for our approach. Moreover, not even the number and the dimension of the subsystems needs to be known and we will also be able to define the $z_i$ in a more general way than we did in this section. All that needs to be assumed in what follows is that a compositional Lyapunov function $V$ of the form \eqref{eq:sglf} exists. In summary, the small-gain theory just presented only serves to show that it is realistic to assume the existence of such a $V$, but the particular subsystem structure does not need to be known for constructing it. Rather, provided that an upper bound for the dimension of the subsystems is known, the resulting compositional form of $V$ will be detected by the training algorithm of the neural network.

\section{Deep neural networks}\label{sec:nn}

This section briefly summarizes the known results about approximation properties of deep neural networks that we are going to use in the subsequent section. A deep neural network is a computational architecture that has several inputs, which are processed through $\ell\ge 1$ hidden layers of neurons. The values in the neurons of the layer with the largest $\ell$ are used in order to compute the output of the network. In this paper, we will only consider feedforward networks, in which the input is processed consecutively through the layers $1$, $2$, \ldots, $\ell$. For our purpose of representing Lyapunov functions, we will use networks with the input vector $x=(x_1,\ldots,x_n)^T\in\R^n$ and a scalar output $W(x;\theta)\in\R$. Here, the vector $\theta\in\R^P$ represents the free parameters in the network that need to be tuned (or ``learned'') in order to obtain the desired output. In our case, the output shall approximate a Lyapunov function, i.e., we want to find $\theta^*$ such that $W(x;\theta^*)\approx V(x)$ for a Lyapunov function $V$ and all $x\in K_n$, where $K_n$ is the compact set on which $V$ shall be computed.
Figure \ref{fig:gen_nn} shows generic neural networks with one and two hidden layers. 
\begin{figure}[htb]
\begin{center}
\includegraphics[width=8cm]{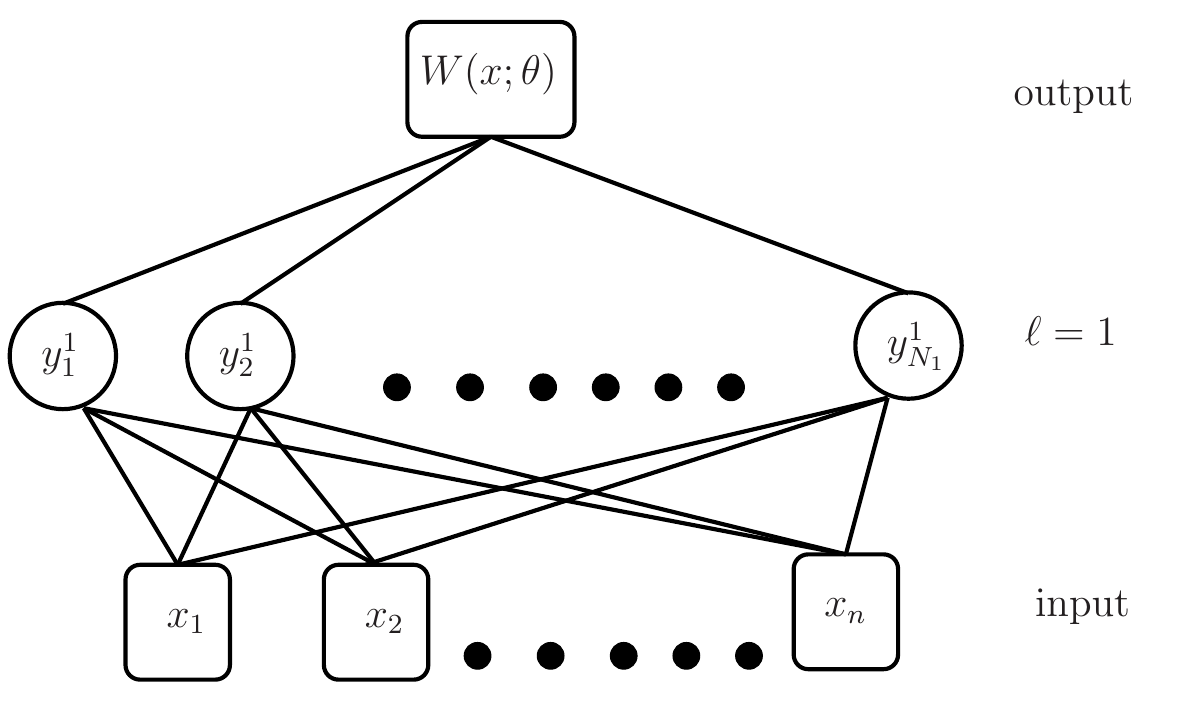}\\[10mm]
\includegraphics[width=8cm]{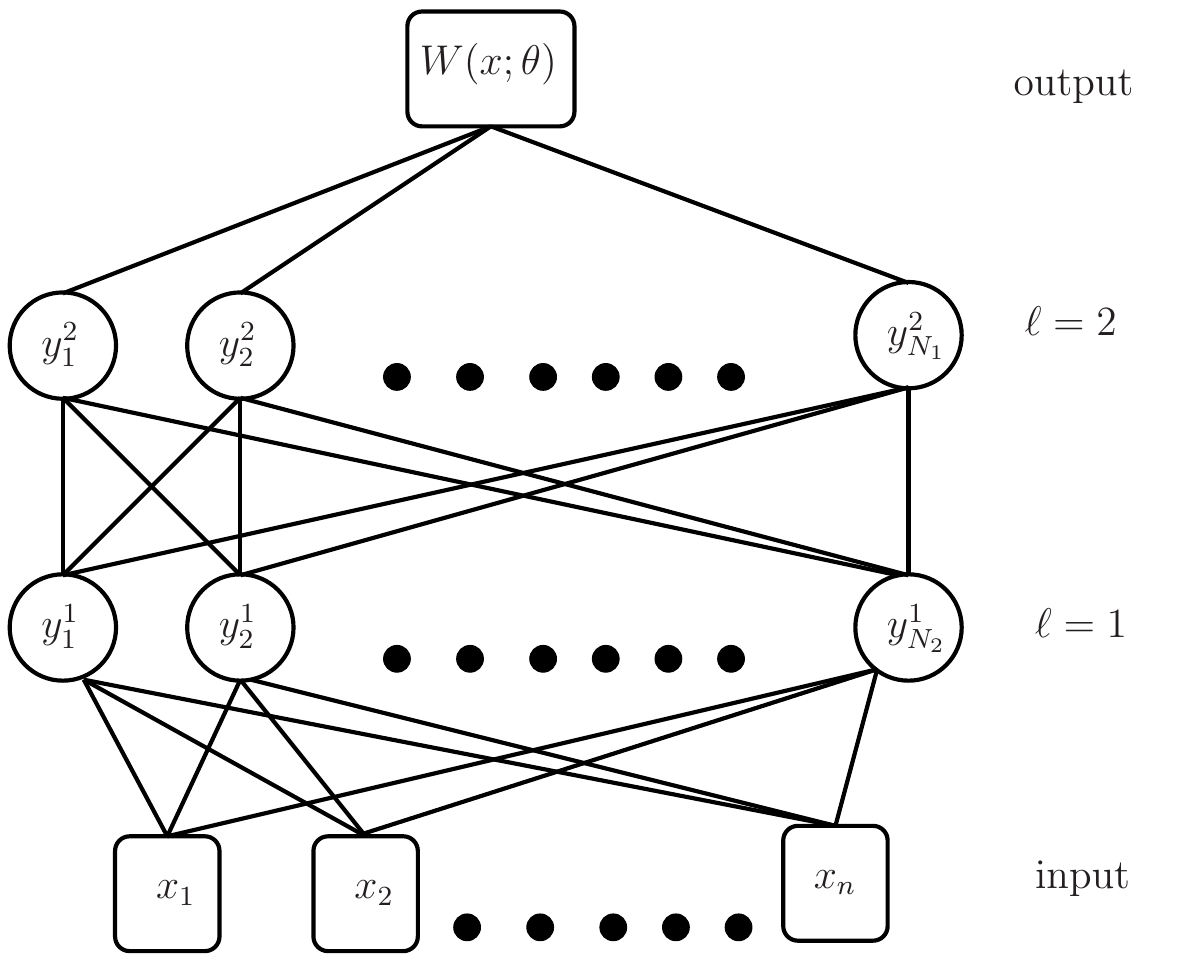}
\end{center}
\caption{Neural network with $1$ and $2$ hidden layers}
\label{fig:gen_nn}
\end{figure}

Here, the lowest layer is the input layer, followed by one or two hidden layers numbered with $\ell$, and the output layer. 
The number $\ell_{\max}$ determines the number of hidden layers, here $\ell_{\max}=1$ or $2$. Each hidden layer consists of $N_\ell$ neurons and the overall number of neurons in the hidden layers is denoted by $N=\sum_{\ell=1}^{\ell_{\max}} N_\ell$. The neurons are indexed using the number of their layer $\ell$ and their position in the layer $k$. Every neuron has a scalar value $y^\ell_k\in\R$ and for each layer these values are collected in the vector $y^{\ell}=(y^\ell_1,\ldots,y^\ell_{N_\ell})^T\in\R^{N_\ell}$. The values of the neurons at the lowest level are given by the inputs, i.e., $y^{0}=x\in\R^n$. The values of the neurons in the hidden layers are determined by the formula
\[ y_k^\ell = \sigma^\ell(w_k^\ell \cdot y^{\ell-1} + b_k^\ell),\] 
for $k=1,\ldots N_\ell$, where $\sigma^\ell:\R\to\R$ is a so called activation function and $w_k^\ell\in\R^{N_{\ell-1}}$, $a_k^\ell,b_k^\ell\in\R$ are the parameters of the layer. In our implementation, below, we will use the softplus activation function $\sigma^\ell(r)=\ln(e^r+1)$ and the linear activation function $\sigma^\ell(r)=r$, the latter for implementing a linear change of coordinates. With $x\cdot y$ we denote the Euclidean scalar product between two vectors $x,y\in\R^n$. In the output layer, the values from the topmost hidden layer $\ell=\ell_{\max}$ are affine linearly combined to deliver the output, i.e., 
\beq W(x;\theta) = \sum_{k=1}^{N_{\ell_{\max}}} a_ky_k^{\ell_{\max}} + c = \sum_{k=1}^{N_{\ell_{\max}}} a_k \sigma^{\ell_{\max}}(w_k^{\ell_{\max}} \cdot y^{\ell_{\max}-1} + b_k^{\ell_{\max}}) + c.\label{eq:output}\eeq
The vector $\theta$ collects all parameters $a_k$, $c$, $w_k^\ell$, $b_k^\ell$ of the network. 

In case of one hidden layer, in which $\ell_{\max}=1$ and thus $y^{\ell_{\max}-1}=y^0=x$, we obtain the closed-form expression
\[ W(x;\theta) = \sum_{k=1}^{N_1} a_k \sigma^1(w_k^1\cdot x + b_k^1) + c.\]
The universal approximation theorem states that a neural network with one hidden layer can approximate all smooth functions arbitrarily well. In its qualitative version, going back to \cite{Cybe89,HoSW89}, it states that the set of functions that can be approximated by neural networks with one hidden layer is dense in the set of continuous functions. In Theorem \ref{thm:univ}, below, we state a quantitative version, given as Theorem 1 in \cite{PMRML17}, which is a reformulation of Theorem 2.1 in \cite{Mhas96}. 

For its formulation consider the compact sets $K_n\subset\R^n$ satisfying \eqref{eq:Kn} on which we want to perform our computation. For a continuous function $g:K_n\to \R$ we define the infinity-norm over $K_n$ as
\[ \|g\|_{\infty,K_n} := \max_{x\in K_n}|g(x)|. \]
We then define the set of functions 
\[ \WW_m^n := \left\{ g\in C^m(K_n,\R) \,\left|\, \sum_{1\le |\alpha|\le m} \| D_\alpha g\|_{\infty,K_n} \le 1\right.\right\} \]
where $C^m(K_n,\R)$ denoted the functions from $K_n$ to $\R$ that are $m$-times continuously differentiable, $\alpha$ are multiindices of length $|\alpha|$ with entries $\alpha_i\in\{1,\ldots,n\}$, $i=1,\ldots,|\alpha|$ and $D_\alpha g = \partial g^{|\alpha|}/\partial \alpha_1\ldots \partial\alpha_{|\alpha|}$ denotes the $m$-th directional derivative with respect to $\alpha$.

\begin{theorem} Let $\sigma:\R\to\R$ be infinitely differentiable and not a polynomial\/\footnote{Polynomials are excluded because in the proof of this theorem it is needed that the derivatives $\sigma^{(k)}$ for all degrees $k\in\N$ do not vanish. See also the discussion after Theorem 1 in \cite{PMRML17}.}. Then, for any $\eps>0$, a neural network with one hidden layer provides an approximation 
\[ \inf_{\theta\in\R^P} \|W(x;\theta) - g(x)\|_{\infty,K_n} \le \eps \]
for all $g\in \WW_m^n$ with a number of $N$ of neurons satisfying
\[ N = \OO\left(\eps^{-\frac{n}{m}}\right)\]
and this is the best possible.
\label{thm:univ}\end{theorem}
\begin{proof} See \cite[Theorem 1]{PMRML17} or \cite[Theorem 2.1]{Mhas96} for this result with $K_n=[-1,1]^n$. The extension to $K_n\subset [-C,C]^n$ is straightforward.
\end{proof}

We note that if $\theta\in\R^P$ realizing the infimum in the inequality in Theorem \ref{thm:univ} exists, then in general it depends on $g$.
Theorem \ref{thm:univ} implies that one can readily use a network with one hidden layer for approximating Lyapunov functions. However, in general the number $N$ of neurons needed for a fixed approximation accuracy $\eps>0$ grows exponentially in $n$, and so does the number of parameters in $\theta$. This means that the storage requirement as well as the effort to determine $\theta$ easily exceeds all reasonable bounds already for moderate dimensions $n$. Hence, this approach also suffers from the curse of dimensionality. In the next section, we will therefore exploit the particular structure of compositional Lyapunov functions in order to obtain neural networks with (asymptotically) much lower $N$.

\section{Neural network structure and complexity analysis}\label{sec:nnlf}

\subsection{The case of known subsystems}\label{sec:known}

For our first result, for fixed $d_{\max}\in\N$ we consider the family of functions
\[ F_1^{d_{\max}} := \left\{ f:\R^n\to\R^n \left| \begin{array}{l} n\in\N,\, f \mbox{ is Lipschitz and } \eqref{eq:sys}
\mbox{ admits a compositional}\\
\mbox{Lyapunov function \eqref{eq:sglf} with }\max_{i=1,\ldots,s} d_i \le d_{\max}
\end{array}\right.\right\}. \]
We assume that for each $f\in F_1^{d_{\max}}$ we know the dimensions $d_i$ and states $z_i$ of the subsystems $\Sigma_i$, $i=1,\ldots,s$, of the corresponding decomposition. For this situation, we use a network with one hidden layer of the form depicted in Figure \ref{fig:Lf_nn_l1}.

\begin{figure}[htb]
\begin{center}
\includegraphics[width=8cm]{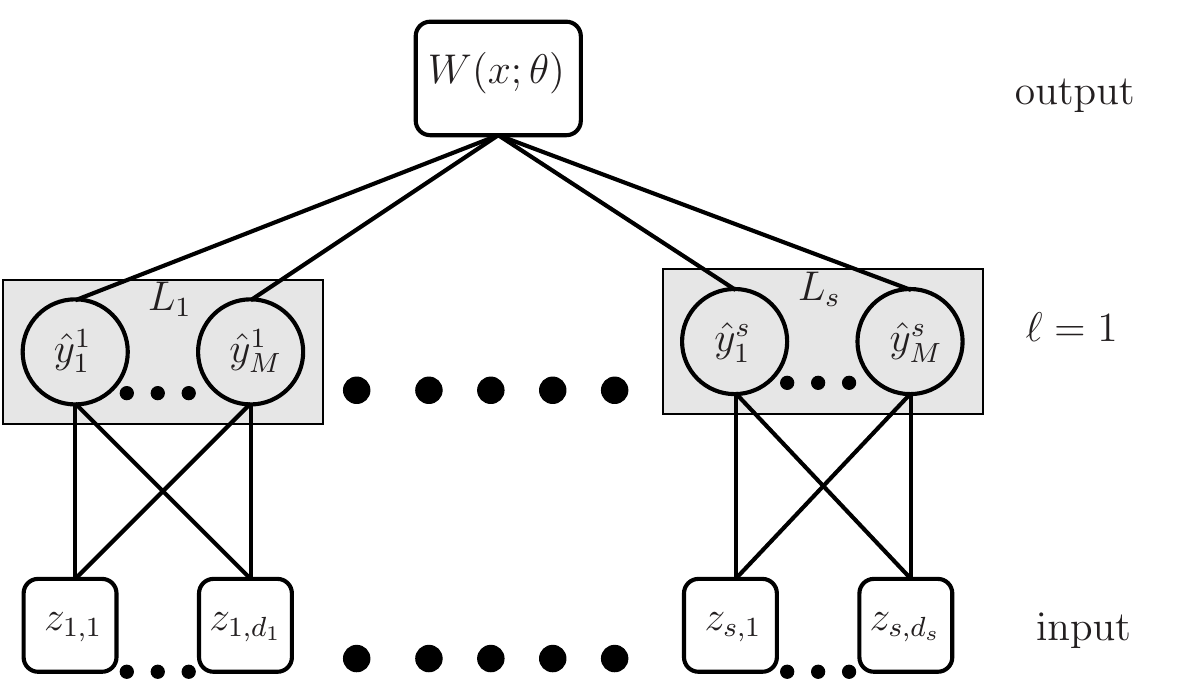}
\end{center}
\caption{Neural network for Lyapunov functions, $f\in F_1^{d_{\max}}$}
\label{fig:Lf_nn_l1}
\end{figure}

In this network, the single hidden layer for $\ell=1$ consists of $s$ sublayers $L_1, \ldots, L_s$. The input of each of the neurons in $L_i$ is the state vector $z_i=(z_{i,1},\ldots,z_{i,d_i})^T$ of the subsystem $\Sigma_i$, which forms a part of the state vector $x$. We assume that every sublayer $L_i$ has $M$ neurons, whose parameters and values are denoted by, respectively, $\hat w_k^i$, $\hat a_k^i$, $\hat b_k^i$ and $\hat y_k^i$, $k=1,\ldots,d_i$. Since $s\le n$, the layer contains $N^1=sM\le nM$ neurons, which is also the total number $N$ of neurons in the hidden layers. The values $\hat y_k^i$ are then given by 
\[ \hat y_k^i = \sigma^1(\hat w_k^i \cdot z_i + \hat b_k^i)\] 
and the overall output of the network is 
\[ W(x;\theta) = \sum_{i=1}^s \sum_{k=1}^{d_i} \hat a_k^i\sigma^1(\hat w_k^i \cdot z_i + \hat b_k^i) + c.\]

\begin{proposition} Given compact sets $K_n\subset\R^n$ satisfying \eqref{eq:Kn}, for each $f\in F_1^{d_{\max}}$ there exist a Lyapunov function $V_f$ such that the following holds. For each $\eps\in(0,1)$ the network depicted in and described after Figure \ref{fig:Lf_nn_l1} with $\sigma^1:\R\to\R$ infinitely differentiable and not polynomial, provides an approximation $\inf_{\theta\in\R^P} \|W(x;\theta) - V_f(x)\|_{\infty,K_n} \le \eps$ for all $f\in F_1^{d_{\max}}$ with a number of $N$ of neurons satisfying

\[ N = \OO\left(n^{d_{\max}+1}\eps^{-d_{\max}}\right).\]
\label{prop:V1}\end{proposition}
\begin{proof} Consider the $C^1$-functions $\widehat V_i$ from \eqref{eq:sglf}. We choose $\mu>0$ maximal such that $\mu \widehat V_i$ lies in $\WW_1^{d_i}$ and set $V_f=\mu V$ with $V=\sum_{i=1}^s \widehat V_i$ from \eqref{eq:sglf}.
We denote the projection of the set $K_n$ corresponding to the variables $z_i$ by $\widehat{K}_i$.
Then, by Theorem \ref{thm:univ} there exist values $\hat a_k^i$, $\hat b_k^i$, $\hat w_k^i$, $\hat c^i$, $k=1,\ldots,d_i$, such that the output 
\[ o_{L_i}(z_i) := \sum_{k=1}^{d_i} \hat a_k^i\sigma^1(\hat w_k^i \cdot z_i + \hat b_k^i) + \hat c^i \]
of each sublayer $L_i$ satisfies
\[ \left\| o_{L_i} - \mu \widehat V_i\right\|_{\infty, \widehat K_i} \le \eps/n \]
for a number of neurons 
\[ M=\OO\left((\eps/n)^{-d_i}\right) \le \; \OO\left((\eps/n)^{-d_{\max}}\right)=\OO\left(n^{d_{\max}}\eps^{-d_{\max}}\right), \]
noting that the inequality used here holds whenever $\eps\le n$, which is satisfied since $\eps<1$. 
Since this is true for all sublayers $L_1$, $\ldots$, $L_s$, by merging the weights $\hat a^i_k$ and $\hat c^i$ into the $a_k$ and $c$ in \eqref{eq:output} we obtain $W(x;\theta) = \sum_{i=1}^s o_{L_i}(z_i)$ and thus
\[ \| W(\cdot;\theta) - V_f\|_{\infty,K_n} \le \sum_{i=1}^s\left\| o_{L_i} - \mu \widehat V_i\right\|_{\infty,\widehat K_i} \le s\eps/n \le \eps \]
with the overall number of neurons $N \le nM = \OO\left(n^{d_{\max}+1}\eps^{-d_{\max}}\right)$.
\end{proof}

\subsection{The case of unknown subsystems}

The approach in the previous subsection requires the knowledge of the subsystems $\Sigma_i$ in order to design the appropriate neural network. This is a rather unrealistic assumption that requires a lot of preliminary analysis in order to set up an appropriate network. Fortunately, there is a remedy for this, which moreover applies to a larger family of systems than $F_1^{d_{\max}}$ considered above. To this end, we consider the family of maps 
\[ F_2^{d_{\max},c} := 
\left\{ f:\R^n\to\R^n \left| \begin{array}{l} n\in\N,\, \mbox{there is an invertible } T\in\R^{n\times n} \mbox{ with}\\ 
\|T\|_\infty\le c, \mbox{ such that } \tilde f \in F_1^{d_{\max}}\\ 
\mbox{for }\tilde f(\tilde x) := Tf(T^{-1}\tilde x)
\end{array}\right.\right\}. \]
Here we make the notational convention that $\tilde x = Tx$ and $\|T\|_\infty$ denotes the matrix norm induced by the vector norm $\|x\|_\infty = \max_{i=1,\ldots,n} |x_i|$. 

Similar as before, we now assume that the transformed vector field $\tilde f$ allows for a compositional Lyapunov function, corresponding to $\tilde s$ subsystems $\widetilde \Sigma_i$, $i=1,\ldots,\tilde s$, with dimensions $\tilde d_i$ and states $\tilde z_i$. However, in
contrast to Section \ref{sec:known}, now we do {\em not} assume that we know the dimensions $\tilde d_i$ and states $\tilde z_i$ of the subsystems $\widetilde \Sigma_i$, and not even their number $\tilde s$. We also do not need to know the coordinate transformation $T$. 
The neural network that we propose for $f\in F_2^{d_{\max},c}$ is depicted in Figure \ref{fig:Lf_nn_l2}. 

\begin{figure}[htb]
\begin{center}
\includegraphics[width=8cm]{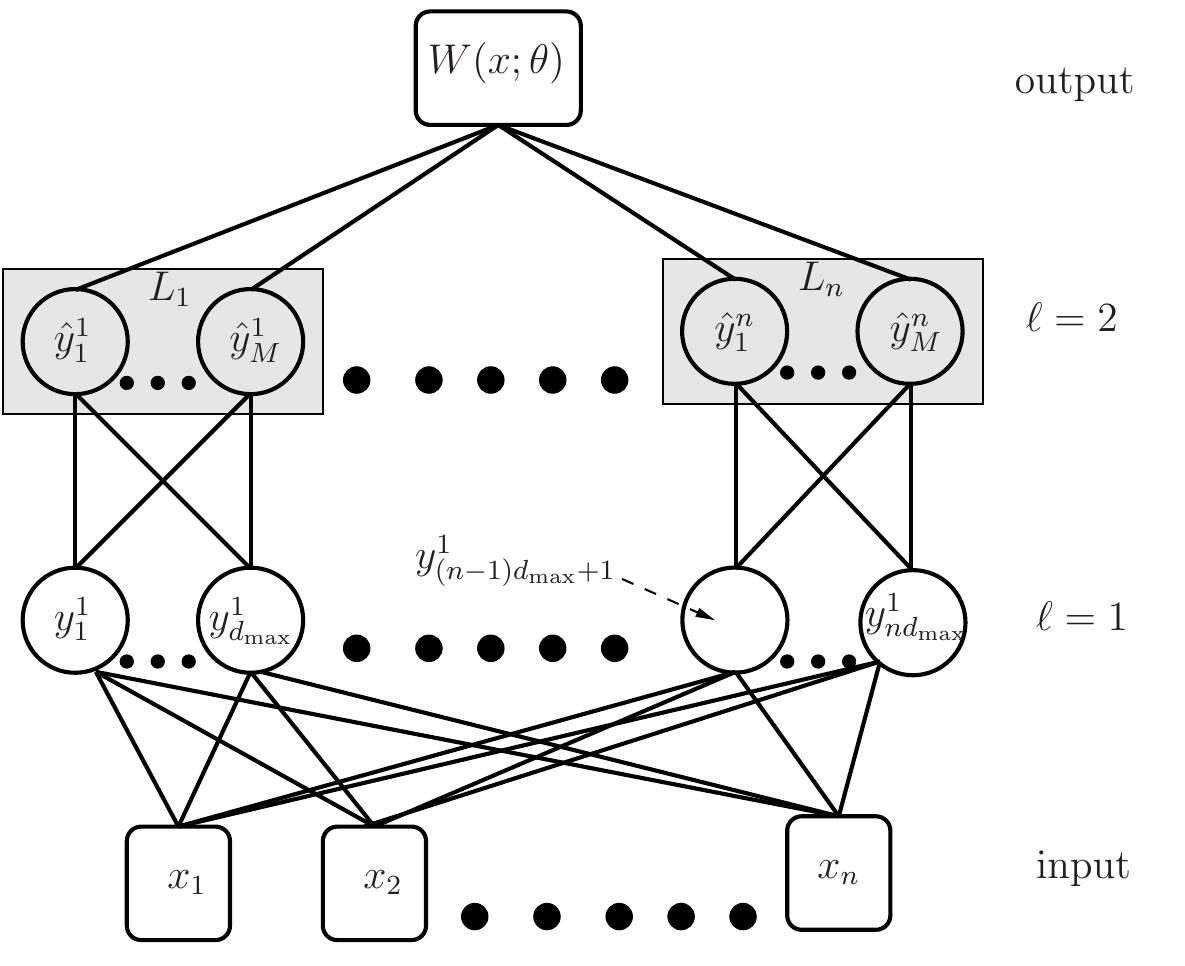}
\end{center}
\caption{Neural network for Lyapunov functions, $f\in F_2^{d_{\max}}$}
\label{fig:Lf_nn_l2}
\end{figure}

Here, we use different activation functions $\sigma^\ell$ in the different levels. While $\sigma^2$ in layer $\ell=2$ is chosen like $\sigma^1$ in Proposition \ref{prop:V1}, in Level $\ell=1$ we use the identity, i.e., the linear activation $\sigma^1(x)=x$. Layer $\ell=2$ consists of $n$ sublayers $L_1$, $\ldots$, $L_n$, each of which has exactly $d_{\max}$ inputs and $M$ neurons. The coefficients and neuron values of each $L_i$ are again denoted with $\hat w_k^i$, $\hat a_k^i$, $\hat b_k^i$ and $\hat y_k^i$, respectively, for $k=1,\ldots,d_{\max}$. The $d_{\max}$-dimensional input of each neuron in $L_i$ is given by 
\[ (y^1_{(i-1)d_{\max}+1},\ldots,y^1_{i d_{\max}})^T =:\bar y^1_i.\] 
We note that this network is a special case of the lower network in Figure \ref{fig:gen_nn}. 

\begin{theorem} Given compact sets $K_n\subset\R^n$ and $c>0$, for each $f\in F_2^{d_{\max},c}$ there exist a Lyapunov function $V_f$ such that the following holds. For each $\eps\in(0,1)$ the network depicted in and described after Figure \ref{fig:Lf_nn_l2} with $\sigma^2:\R\to\R$ infinitely differentiable and not polynomial in layer $\ell=2$ and $\sigma^1(x)=x$ in layer $\ell=1$, provides an approximation $\inf_{\theta\in\R^P} \|W(x;\theta) - V_f(x)\|_{\infty,K_n} \le \eps$ for all $f\in F_2^{d_{\max}}$ with a number of $N$ of neurons satisfying

\[ N = \OO\left(n d_{\max}+n^{d_{\max}+1}\eps^{-d_{\max}}\right).\]
\label{thm:V2}\end{theorem}
\begin{proof} Let $\tilde d_i$ be the (unknown) dimensions of the subsystems $\widetilde\Sigma_i$, $i=1,\ldots,\tilde s$, and $p_i = 1+ \sum_{k=1}^{i-1} \tilde d_k$ the first index of the variables $\tilde z_i$ of $\widetilde \Sigma_i$, i.e., $\tilde z_i=(\tilde x_{p_i},\ldots,$ $\tilde x_{p_{i+1}-1})^T$ for $\tilde x =Tx$. Using the notation from above and the fact that $\sigma^1(x)=x$, the values of the inputs $y^1_k$ of sublayer $\ell=1$ are given by 
\[ y^1_k = w_k^1 \cdot x + b_k^1. \]
Hence, by choosing $b_k^1=0$ and $w_k^1$ as the transpose of the \mbox{$j$-th} row of $T$, we obtain $y^1_k = \tilde x_j$. Hence, assigning the $b_k^1$ and $w_k^1$ this way for $k=(i-1)d_{\max}+j-p_i+1$, $j=p_i,\ldots,p_{i+1}-1$, and $i=1,\ldots,\tilde s$, and setting the remaining $b_k^1$ and $w_k^1$ to 0, we obtain 
\[ \bar y^1_i = \left(\begin{array}{c} \tilde z_i\\ 0 \\ \vdots \\ 0\end{array}\right)\]
for $i=1,\ldots,\tilde s$, where the number of the zeros equals $d_{\max}-d_i$. The inputs for the remaining sublayers $L_{\tilde s+1}$, $\ldots$, $L_n$ are $0$ since the corresponding $w^1_k$ and $b^1_k$ are set to $0$. For this choice of the parameters of the lower layer, each sublayer $L_i$ of the layer $\ell=2$ receives the transformed subsystem states $\tilde z_i$ (and a number of zeros) as input, or the input is $0$.

Since the additional zero-inputs do not affect the properties of the network, the upper part of the network, consisting of the hidden layer $\ell=2$ and the output, has exactly the structure of the network used in Proposition~\ref{prop:V1}. We can thus apply this proposition on the sets $\widetilde K_n = TK_n$ to the upper part of the network and obtain that it can realize a function $W(\tilde z;\theta)$ that approximates a Lyapunov function $\widetilde V$ for $\tilde f$ in the sense of Proposition \ref{prop:V1}. Note that since $\|T\|_\infty \le c$ and $K_n$ satisfies \eqref{eq:Kn}, we have that $\widetilde K_n = TK_n \subset [-cC,cC]^n$, hence $\widetilde K_n$ also satisfies \eqref{eq:Kn}. The fact that the constant bounding $\widetilde K_n$ is different now does not pose a problem when we apply Proposition~\ref{prop:V1}, as it only leads to different constants in the resulting term, which vanish in the $\OO$-term.

As the lower layer realizes the coordinate transformation $\tilde x = Tx$, the overall network $W(x;\theta)$ then approximates the function $V(x) := \widetilde V(Tx)$. By means of the invertibility of $T$ and the chain rule one easily checks that this is a Lyapunov function for $f$. The claim then follows since the number of neurons $N^2$ in the upper layer is equal to that given in Proposition \ref{prop:V1}, while that in the lower layer equals $N^1=nd_{\max}$. This leads to the overall number of neurons given in the theorem.
\end{proof}

\begin{remark} We note that the theorem remains true if the number of sublayers in the neural network from Figure \ref{fig:Lf_nn_l2} is reduced from $n$ to $n'$, as long as $n'\ge \tilde s$. Setting $n'=n$ ensures this inequality, but if a priori information about $\tilde s$ is available, then this could be used in order to reduce the size of the network.
\label{rem:n}\end{remark}

\section{Training the network}\label{sec:training}

For training the network in order to actually compute a Lyapunov function we need to specify a loss function $L:\R\times\R^n\times\R^n \to \R$. Training then consists of finding parameters $\theta$ such that 
\beq \frac1m \sum_{i=1}^m L\left(W(x^{(i)};\theta),D W(x^{(i)};\theta),x^{(i)}\right) \label{eq:loss}\eeq
becomes minimal, where $x^{(i)}\in K_n$ are the elements of a test dataset, which we refer to as test points. In our numerical results in the next section we always use $K_n=[-1,1]^n$ and the test points $x^{(i)}$ are chosen randomly and uniformly distributed from $K_n$. 

Note that in contrast to many other problems in deep learning the loss function $L$ also depends on the values of the derivative of $W$ with respect to $x$ in the test points, which we denote by $DW(x^{(i)},\theta)$. This is needed because in order to determine whether $W$ is a Lyapunov function, its derivative is needed. For minimizing the expression \eqref{eq:loss} a stochastic gradient algorithm can be used, which is standard in deep learning \cite{Bott10,BoCN18}. Details are specified in the following section.

The main work is now to design the loss function such that minimizing \eqref{eq:loss} w.r.t.\ $\theta$ yields a Lyapunov function. To this end, a straightforward idea is to express the Lyapunov function property as a partial differential equation (PDE) and follow the approaches in the literature for solving PDEs with neural networks mentioned in the introduction. A simple PDE that is suitable for this purpose is the Zubov-type equation
\beq DW(x;\theta)f(x) = - \|x\|^2,\label{eq:lfpde}\eeq
similar PDEs have been used or discussed, e.g., in \cite{CaGW00,Gies07,KiNB82,Zubo64}. This PDE needs to be complemented by suitable boundary conditions, which in the PDE setting (with $x=0$ being the equilibrium of interest) are of the form 
\[ W(0,\theta)  = 0 \quad \mbox{and} \quad W(x;\theta)>0 \mbox{ for all } x \in K_n\setminus\{0\}. \]
However, in this form the boundary conditions are difficult to be implemented numerically: the equality condition $W(0,\theta)  = 0$ is difficult because it is only given in a single point, while the strict ``$>$'' condition is difficult because numerically only ``$\ge$'' can be enforced directly. To resolve these problems, we replace the boundary conditions above by the stronger conditions
\beq \alpha_1(\|x\|) \le  W(x;\theta) \le \alpha_2(\|x\|) \mbox{ for all } x \in K_n, \label{eq:bcs}\eeq
with $\alpha_1,\alpha_2\in\KK$. Of course, the functions $\alpha_i$ have to be chosen appropriately in order to allow for the existence of a solution of \eqref{eq:lfpde} that satisfies \eqref{eq:bcs}. However, it follows from \cite{Zubo64} that if a Lyapunov function on $K_n$ exists, then it is always possible to find such $\alpha_i$. Loosely speaking, $\alpha_1$ must be sufficiently flat while $\alpha_2$ must be sufficiently steep. In case $x=0$ is exponentially stable and $f$ is continuously differentiable, one can choose the $\alpha_i$ as quadratic functions $\alpha_i(r) = c_ir^2$ with $c_1>0$ sufficiently small and $c_2>0$ sufficiently large, cf.\ \cite[Theorem 4.14]{Khal96}. 

Given the vector field $f$ from \eqref{eq:sys}, the loss function $L$ is now defined as 
\beq L(w,p,x) := \left(p f(x) + \|x\|^2\right)^2 + \nu \left(\,\left([w-\alpha_1(\|x\|)]_-\right)^2 + \left([w-\alpha_2(\|x\|)]_+\right)^2\,\right), \label{eq:pdeloss}\eeq
where $[a]_-:= \min\{a,0\}$, $[a]_+:= \max\{a,0\}$, and $\nu>0$ is a weighting parameter (chosen as $\nu=1$ in all our numerical examples in the next section). One easily checks that for this $L$ the expression \eqref{eq:loss} is always $\ge 0$ and equals $0$ if and only if \eqref{eq:lfpde} and \eqref{eq:bcs} are satisfied for all test points $x^{(i)}$. Conversely, if a Lyapunov function exists for which the bounds \eqref{eq:bcs} are feasible, and if this Lyapunov function can be represented by neural network under consideration, the minimizing \eqref{eq:loss} w.r.t.\ $\theta$ will result in the optimal value of \eqref{eq:loss} being $0$.

Unfortunately, while this approach works in principle, it is not necessarily compatible with the complexity analysis from the previous section. The reason is that when a Lyapunov function with the particular small gain structure \eqref{eq:sglf} exists, it may not be a solution of \eqref{eq:lfpde}, \eqref{eq:bcs}. As a consequence, while a solution of \eqref{eq:lfpde}, \eqref{eq:bcs} may exist, it may not be representable by the neural network structure from Figure \ref{fig:Lf_nn_l1} or Figure \ref{fig:Lf_nn_l2}. Hence, with the choice of $L$ from \eqref{eq:pdeloss}, it may not be possible to exploit the low computational complexity provided by this particular network structure. The result depicted in Figure \ref{fig:Lyapunov_2d_eq}, below, shows that this indeed happens. 

Hence, we need to provide more flexibility to our approach, which we can do by enlarging the set of minima of the loss function. To this end, note that \eqref{eq:lfpde} is actually a much too strong condition. Requiring the partial differential inequality (PDI)
\beq DW(x;\theta)f(x) \le - \|x\|^2,\label{eq:lfpdi}\eeq
instead of \eqref{eq:lfpde}, will also yield a Lyapunov function. While one may argue that the bound ``$-\|x\|^2$'' on the derivative is somewhat arbitrary here, it is easily seen that by appropriate rescaling any Lyapunov function can be modified such that this bound holds. Hence, modifying the right hand side of \eqref{eq:lfpdi} does not provide more flexibility (but, of course, it affects the set of $\alpha_i$ for which \eqref{eq:lfpdi} and \eqref{eq:bcs} together are feasible). 

Incorporating \eqref{eq:lfpdi} instead of \eqref{eq:lfpde} in the loss function $L$ leads to the expression
\beq L(w,p,x) := \left(\left[p f(x) + \|x\|^2\right]_+\right)^2 + \nu \left(\,\left([w-\alpha_1(\|x\|)]_-\right)^2 + \left([w-\alpha_2(\|x\|)]_+\right)^2\,\right). \label{eq:pdiloss}\eeq
One easily checks that for this $L$ the expression \eqref{eq:loss} is again always $\ge 0$, but now it equals $0$ if and only if \eqref{eq:lfpdi} and \eqref{eq:bcs} are satisfied for all test points $x^{(i)}$. As Example \ref{eq:Lyapunov_2d_le} and Figure \ref{fig:Lyapunov_2d_le}, below, show, this indeed allows to use the network structure from the previous section and it also allows for solving higher dimensional problems, see Example \ref{eq:Lyapunov_10d}.

\section{Numerical examples}\label{sec:numerics}

We illustrate the proposed method with two examples, a low-dimensional one that shows that the the loss function \eqref{eq:pdiloss} is in general preferable over \eqref{eq:pdeloss} and a larger one that shows the ability of the method to work in find Lyapunov functions in higher dimensions. All computations were performed with Python 3.7.0 and TensorFlow 2.1.0 \cite{tens15} on a MacBook Pro (2017, 2.3 GHz Intel Core i5) running macOS Mojave (10.14.6). The python code and the trained networks are available from \href{http://numerik.mathematik.uni-bayreuth.de/\%7Elgruene/DeepLyapunov/}{numerik.mathematik.uni-bayreuth.de/$\sim$lgruene/DeepLyapunov/}. 

Our first example considers a two-dimensional example that has a compositional Lyapunov function consisting of two one-dimensional functions. It is given by 
\beq
\begin{array}{rcl}
\dot x_1 & = & -x_1-10x_2^2\\
\dot x_2 & = & -2x_2.
\end{array}
\label{eq:Lyapunov_2d_le}
\eeq
Using the Lyapunov-function candidate $V(x)=x_1^2 + x_2^2 + 13 x_2^4$, one computes
\[ DV(x)f(x) = -2x_1^2 - 20x_1x_2^2 - 4x_2^2 - 104 x_2^4. \]
Since 
\[ -x_1^2 - 20x_1x_2^2 - 104 x_2^4 \le -x_1^2 - 20x_1x_2^2 - 100 x_2^4 = -(x_1+10x_2^2)^2 \le 0, \]
we obtain $DV(x)f(x) \le -x_1^2-4x_2^2 \le -\|x\|^2$. Hence, $V$ is a Lyapunov function and it is obviously of the compositional form \eqref{eq:sglf} with $z_1=x_1$ and $z_2=x_2$. 

It should thus be possible to compute a Lyapunov function with the neural network from Figure \ref{fig:Lf_nn_l2}. 
It turns out that using the loss function \eqref{eq:pdiloss} (with $\alpha_1(r)=0.1r^2$ and $\alpha_2(r)=10r^2$) this is indeed possible. 
Here we used the network structure from Figure \ref{fig:Lf_nn_l2} with $n=2$ and $d_{\max}=1$, with the layers $L_1$ and $L_2$ consisting of 128 neurons, each, and softplus activation functions $\sigma^2(r)=\ln(e^r+1)$, resulting in 775 trainable parameters. The training was performed with 200\,000 test points\footnote{In all examples, the number of test points was increased until the results produced satisfactory Lyapunov functions.}, optimizing with batch size 32 using the Adam optimizer implemented in TensorFlow. The optimization was terminated when the accuracy for the final function $W(\cdot,\theta^*)$ satisfied\footnote{Since $L$ consists of squared penalization terms, $err_1$ is effectively the squared weighted $\|\cdot\|_2$-norm of the penalization terms.}
\[ err_1 := \frac1m \sum_{i=1}^m L\left(W(x^{(i)},\theta^*),D W(x^{(i)},\theta^*),x^{(i)}\right)  < 10^{-6} \]
and 
\[ err_\infty :=\max_{i=1,\ldots,m} L\left(W(x^{(i)},\theta^*),D W(x^{(i)};\theta^*),x^{(i)}\right) < 10^{-6},\]
which was reached after 6 epochs in the run documented here.\footnote{As the test points are random, the results of the training optimization are random, too. The error tolerance $10^{-6}$ was sometimes reached already after 4 epochs and sometimes it was not reached until epoch 20. In all successful runs, the resulting Lyapunov was very similar to the one depicted here.} The time needed for the optimization was 48s.
Figure \ref{fig:Lyapunov_2d_le} shows the computed approximate Lyapunov function $W(\cdot,\theta^*)$ as a solid surface along with its derivative along the vector field $DW(x;\theta^*)f(x)$ as a wireframe, shown from two different angles. The graphs illustrate that the method was successful. 

\begin{figure}[htb]
\begin{center}
\includegraphics[width=6cm]{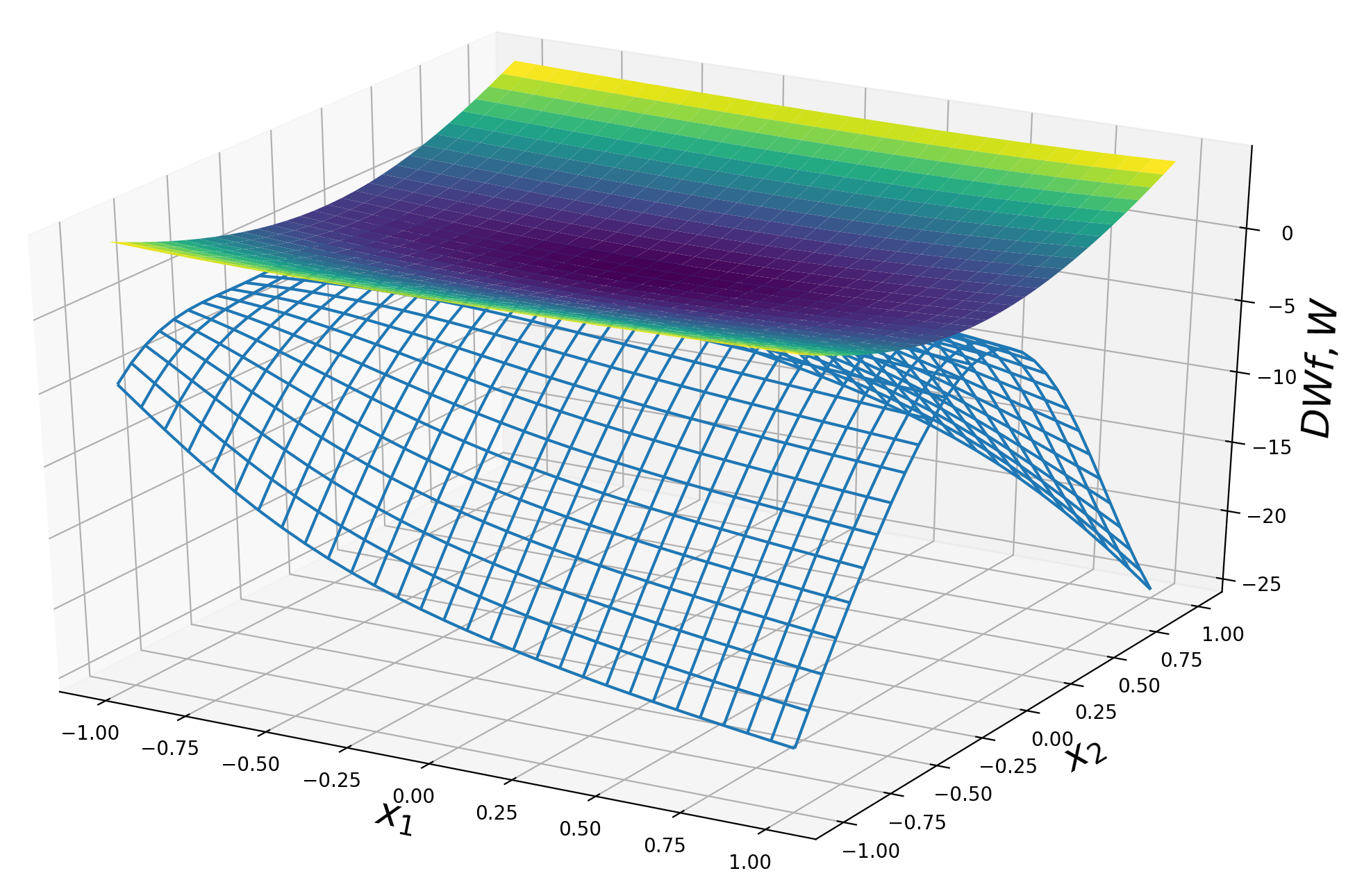}
\includegraphics[width=6cm]{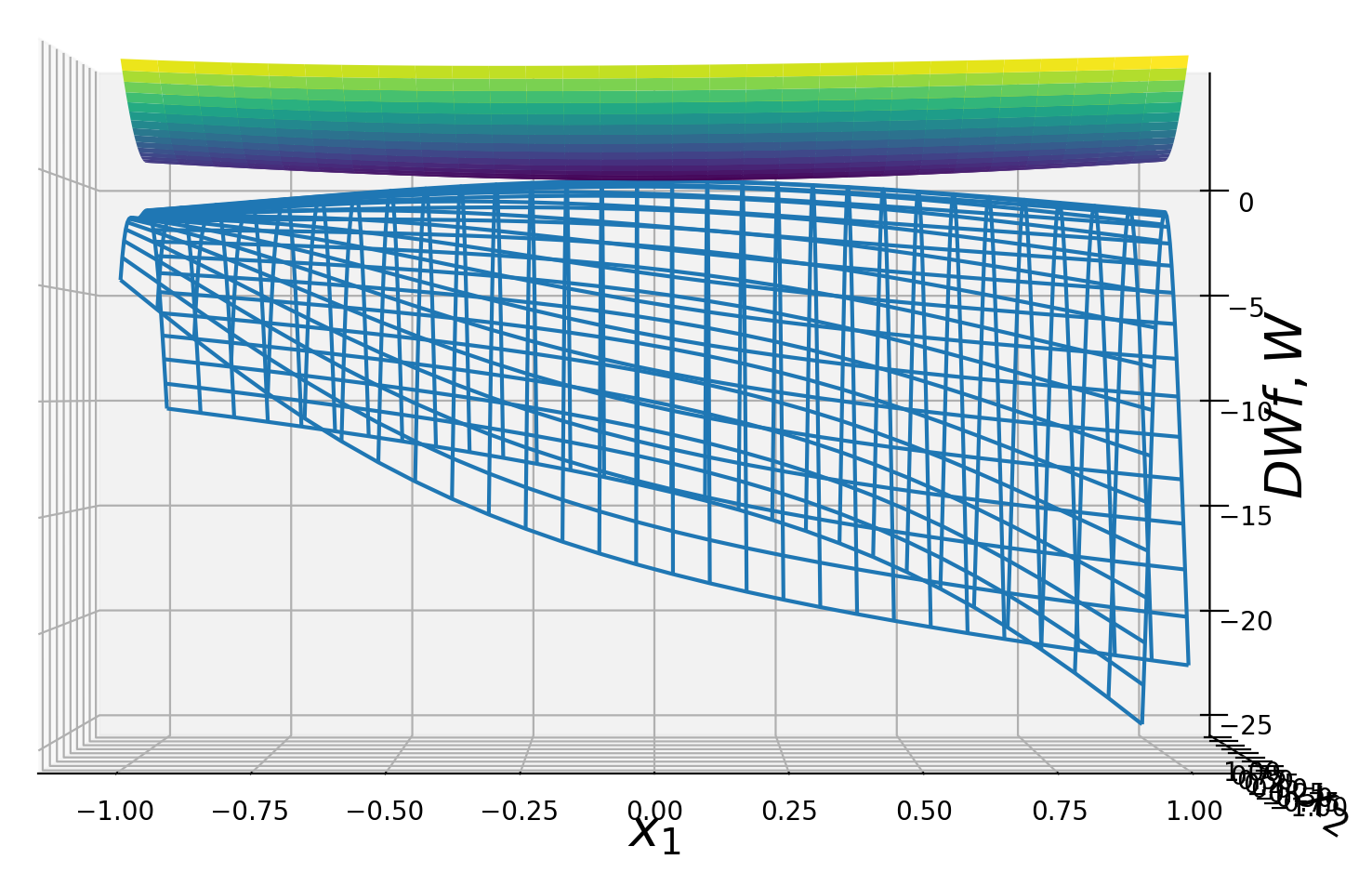}
\end{center}
\caption{Approximate Lyapunov function $W(\cdot;\theta^*)$ (solid) and its orbital derivative $DW(\cdot;\theta^*)f$ (mesh) for Example \eqref{eq:Lyapunov_2d_le} computed with loss function \eqref{eq:pdiloss}}
\label{fig:Lyapunov_2d_le}
\end{figure}

In contrast to this, performing the computation with the same parameters but with loss function \eqref{eq:pdeloss} fails. As Figure \ref{fig:Lyapunov_2d_eq} shows, the derivative $DW(x;\theta^*)f(x)$ (shown as a wireframe) obviously not satisfy the equation $DW(x;\theta^*)f(x)=-\|x\|^2$. This is also visible in the values
\[ err_1   = 1.363842 \;\; \mbox{ and } \;\; err_\infty = 3.110839\]
that were reached after 20 epochs\footnote{In all runs these error values did not change significantly anymore after epoch 15. In some runs the resulting function had a different shape, but in all cases it visibly violated the required inequalities.}. While this alone would not be a problem (as long as $DW(x;\theta^*)f(x)$ is still negative definite), the inability to meet this equation has the side effect that the optimization also does not enforce the inequalities \eqref{eq:bcs}. As a consequence, the minimum of the computed function is not located in the equilibrium at the origin, as the lateral view on the right of Figure \ref{fig:Lyapunov_2d_eq} shows. This is because it is more difficult to represent a Lyapunov function satisfying $DV(x)f(x)=-\|x\|^2$ with the network structure from 
Figure \ref{fig:Lf_nn_l2}. While this example does, of course, not exclude that the loss function \eqref{eq:pdeloss} works for other parameters, it provides evidence that the advantage in computational complexity offered by our approach is more easily exploited using the loss function \eqref{eq:pdiloss}. Moreover, it illustrates the effect when the parameter $d_{\max}$ underestimates the maximal dimension of the subsystems.

\begin{figure}[htb]
\begin{center}
\includegraphics[width=6cm]{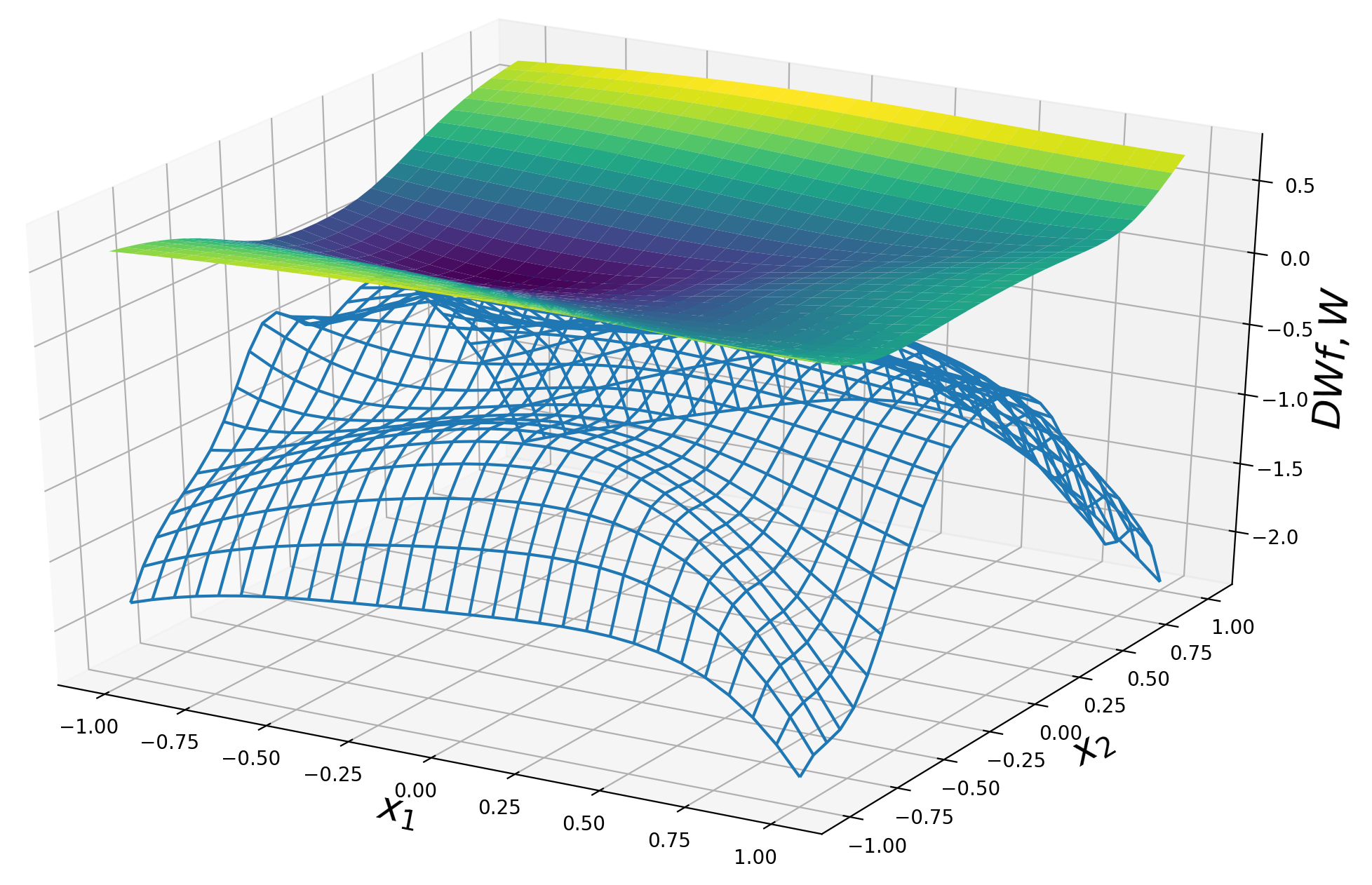}
\includegraphics[width=6cm]{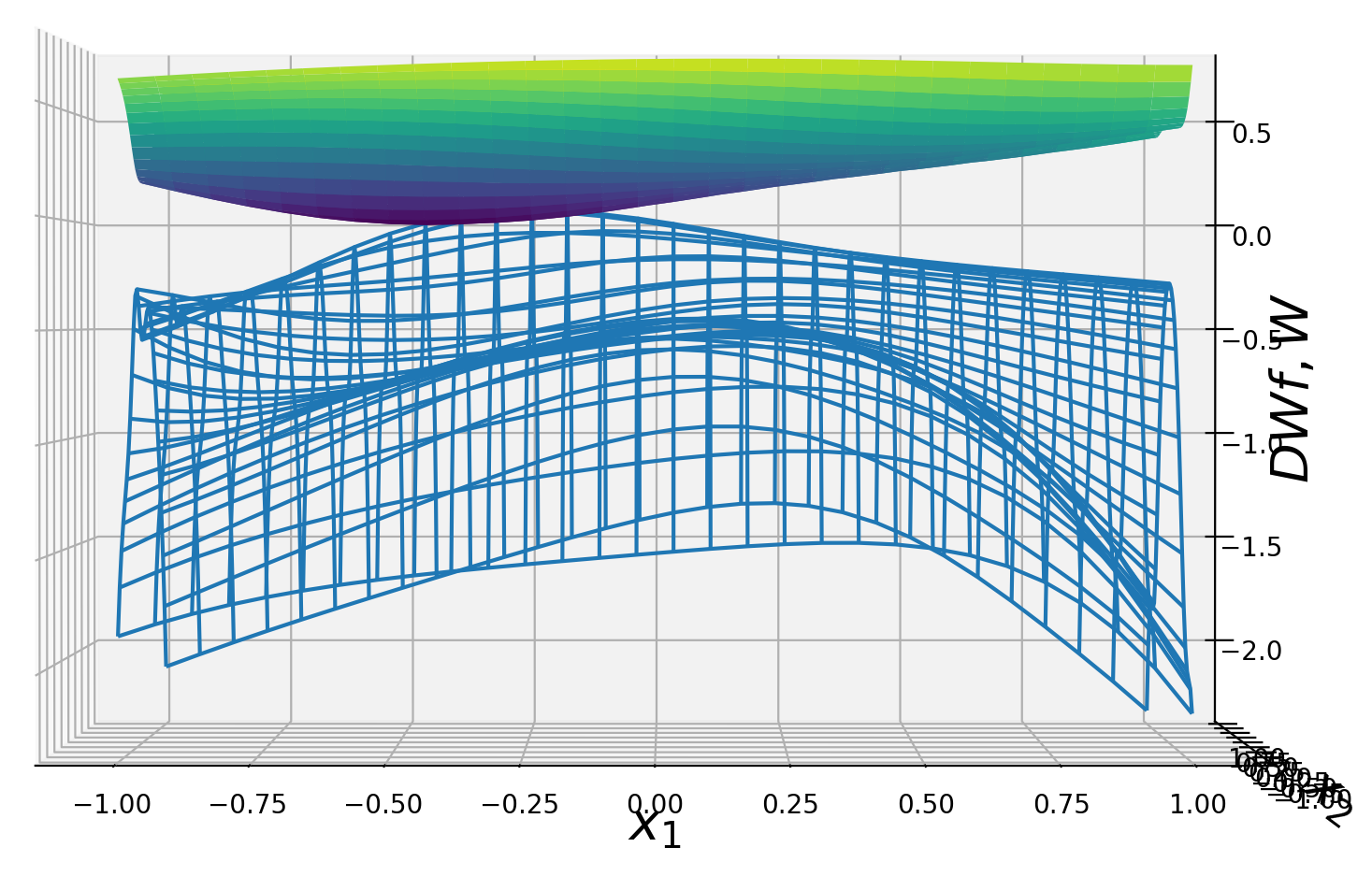}
\end{center}
\caption{Attempt to compute a Lyapunov function $W(\cdot;\theta^*)$ (solid) with its orbital derivative $DW(\cdot;\theta^*)f$ (mesh) for Ex.\ \eqref{eq:Lyapunov_2d_le} with loss function \eqref{eq:pdeloss}}
\label{fig:Lyapunov_2d_eq}
\end{figure}

In our second example we illustrate the capability of our approach to handle higher dimensional systems and to determine the subspaces for the compositional representation of $V$. To this end we consider a 10-dimensional example of the form
\beq 
\dot x = f(x) := T^{-1}\hat f(Tx). \label{eq:Lyapunov_10d}
\eeq
with vector field $\hat f:\R^{10}\to\R^{10}$ given by
\[ \hat f(x) = \left(\begin{array}{l}
-     x_1 + 0.5x_2 - 0.1x_9^2\\
      - 0.5x_1 -     x_2\\
      -     x_3 + 0.5x_4 - 0.1x_1^2\\
      - 0.5x_3 -     x_4\\
      -     x_5 + 0.5x_6 + 0.1x_7^2\\
      - 0.5x_5 -     x_6\\
      -     x_7 + 0.5x_8\\
      - 0.5x_7 -     x_8\\
      -     x_9 + 0.5x_{10}\\
      - 0.5x_9 -     x_{10}+ 0.1x_2^2
      \end{array}\right)
      \]
      
One easily sees that this system consists of five two-dimensional asymptotically stable linear subsystems that are coupled by four nonlinearities with small gains. It is thus to be expected that on $K_{10}=[-1,1]^{10}$ the system is asymptotically stable and a Lyapunov function can be computed using the network from Figure \ref{fig:Lf_nn_l2} five two-dimensional sublayers $L_1,\ldots,L_5$. The coordinate transformation $T\in\R^{10\times 10}$ is given by the (randomly generated) matrix
\[ T = \left(\begin{array}{rrrrrrrrrr}
-\frac15 & -\frac3{10} & \frac12 & -\frac45 & \frac45 & \frac25 & \frac7{10} & \frac7{10} & -1 & \frac45\\[1ex]
\frac15 & 1 & \frac9{10} & \frac45 & -\frac1{10} & \frac35 & -\frac3{10} & \frac12 & \frac45 & -\frac3{10}\\[1ex]
-\frac3{10} & \frac3{10} & \frac25 & -\frac25 & 0 & -\frac35 & \frac3{10} & \frac35 & 1 & -\frac12\\[1ex]
-\frac7{10} & -\frac1{10} & -\frac35 & -\frac15 & -\frac35 & \frac25 & \frac1{10} & -\frac1{10} & \frac1{10} & -\frac35\\[1ex]
\frac1{10} & -\frac35 & -\frac9{10} & -\frac7{10} & -\frac15 & -\frac1{10} & \frac1{10} & \frac15 & 0 & -\frac45\\[1ex]
\frac35 & \frac9{10} & -\frac15 & 1 & \frac25 & \frac12 & 0 & -\frac1{10} & -\frac25 & 0\\[1ex]
-1 & 1 & \frac7{10} & \frac35 & -\frac45 & -\frac45 & 0 & -\frac15 & -\frac15 & \frac7{10}\\[1ex]
-\frac9{10} & \frac45 & \frac15 & 1 & -\frac45 & \frac25 & -\frac3{10} & \frac7{10} & \frac15 & -\frac45\\[1ex]
\frac35 & -\frac1{10} & -\frac25 & -\frac12 & -\frac3{10} & -\frac1{10} & -\frac7{10} & 1 & \frac45 & -\frac3{10}\\[1ex]
0 & -1 & -\frac1{10} & \frac25 & -\frac3{10} & -\frac1{10} & -\frac15 & \frac7{10} & -\frac1{10} & \frac45
\end{array}\right).
\]

We have computed a Lyapunov function for this system for the loss function \eqref{eq:pdiloss} with $\alpha_1(r)=0.2r^2$ and $\alpha_2(r)=10r^2$. 
We used the network structure from Figure \ref{fig:Lf_nn_l2} and Remark \ref{rem:n} with $n'=5$ and $d_{\max}=2$, with the layers $L_1,\ldots,L_5$ consisting of 128 neurons, each, leading to 2671 trainable parameters. The training was performed with 400\,000 test points, optimizing over 13 epochs. As for the 2d example, we used batch size 32, the Adam optimizer implemented in TensorFlow, and softplus activation functions $\sigma^2$. The time needed for the training was 266s\footnote{The time for the evaluation of $W(x;\theta^*)$ in 10\,000 test points takes 0.3s, while the evaluation of the derivative $DW(x;\theta^*)$ in 10\,000 test points takes 0.1s.}  and the resulting function satisfies the inequalities
\[ err_1 < 10^{-6}, \quad err_\infty < 10^{-6}.\]
Figures \ref{fig:Lyapunov_10d_x28} and \ref{fig:Lyapunov_10d_x910} show the resulting function $W(\cdot;\theta^*)$ (solid) and its derivative along $f$ (wireframe) on the $(x_2,x_8)$-plane and the $(x_9,x_{10})$-plane, respectively. The remaining components of $x$ were set to $0$ in both figures. Figure \ref{fig:Lyapunov_10d_tra} shows the value of $W(\cdot;\theta^*)$ along three trajectories of \eqref{eq:Lyapunov_10d} (computed numerically using the ode45-routine from matlab). It shows the strict decrease that is expected from a Lyapunov function.

\begin{figure}[htb]
\begin{center}
\includegraphics[width=6cm]{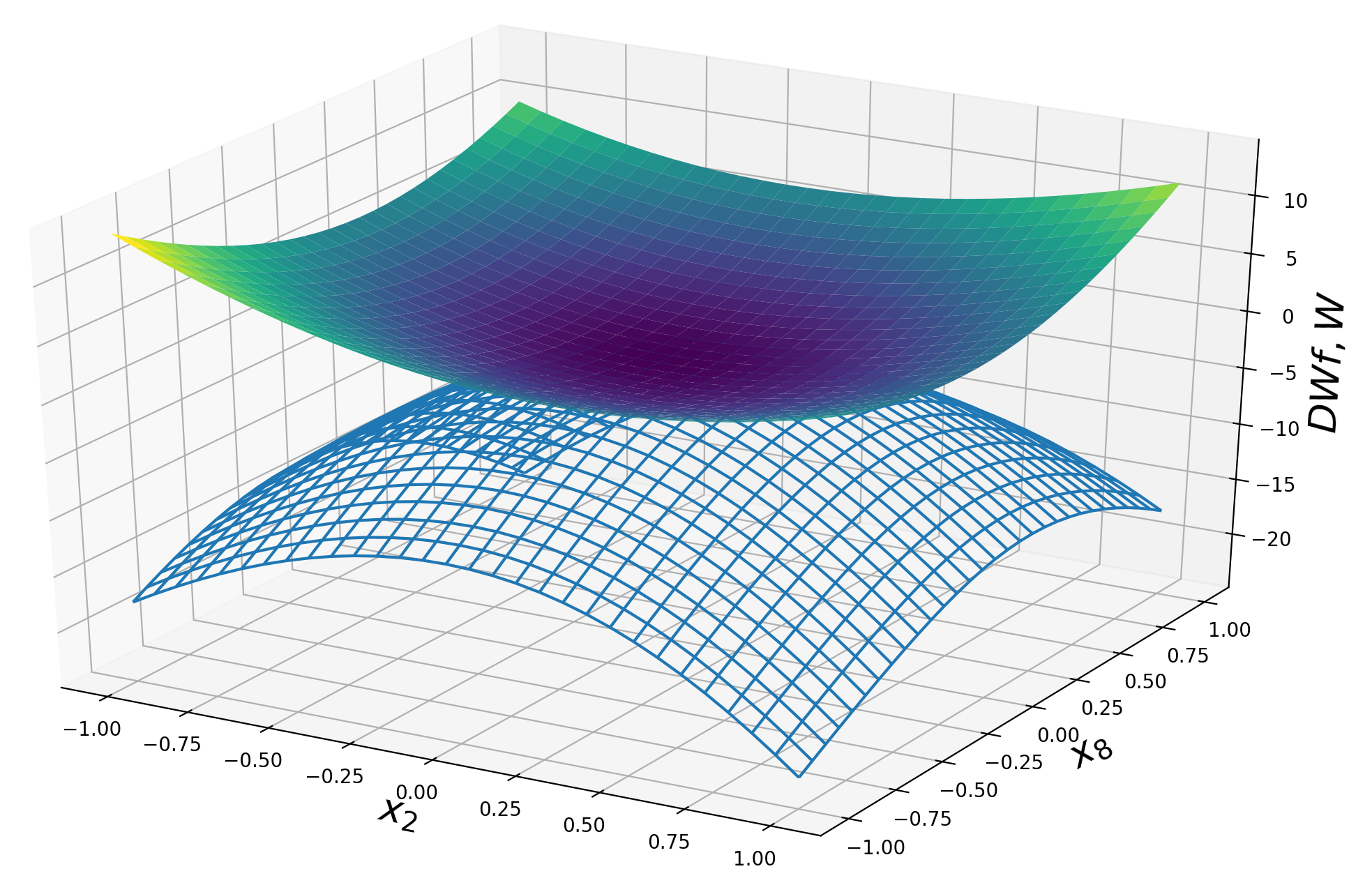}
\includegraphics[width=6cm]{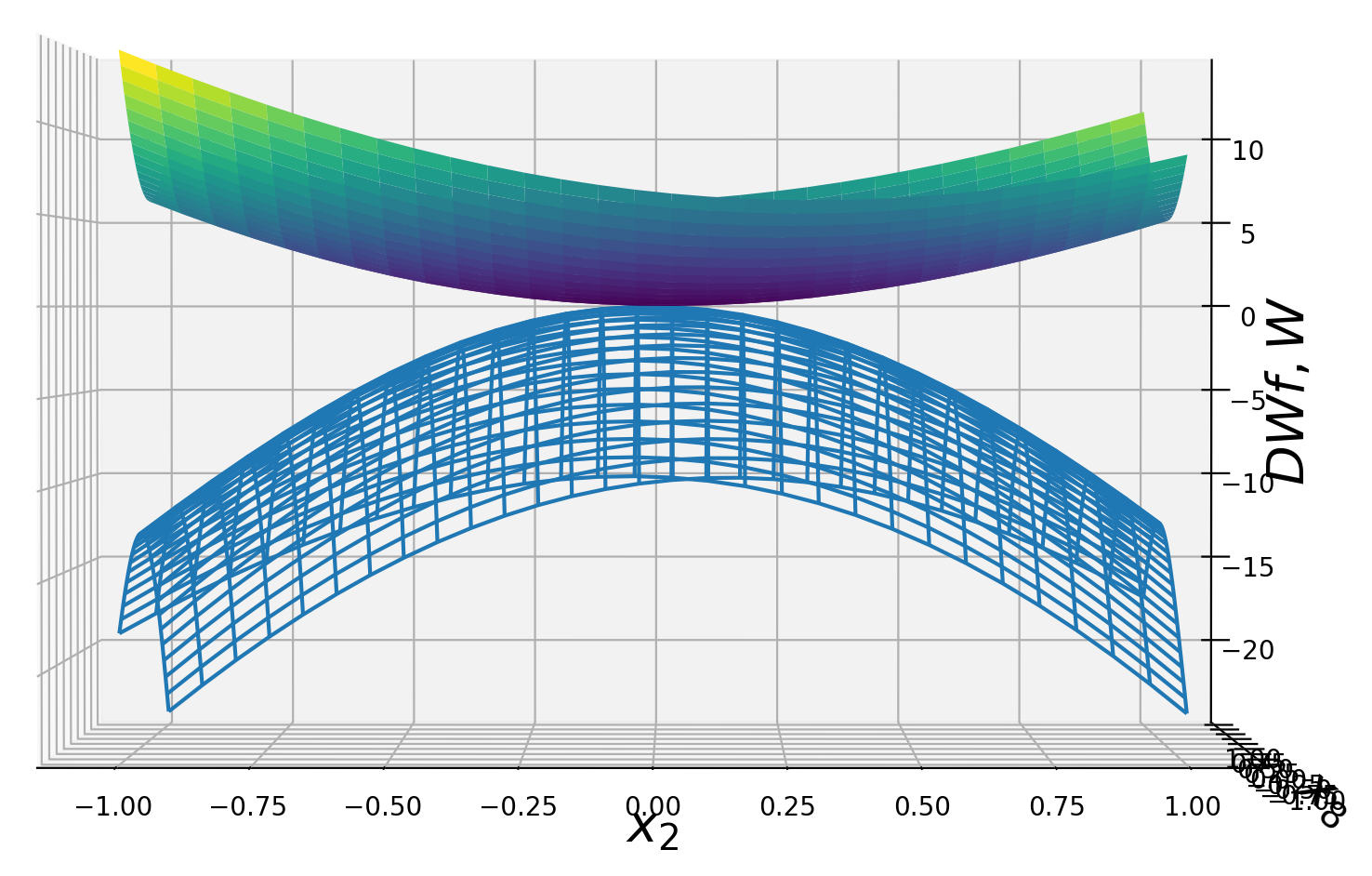}
\end{center}
\caption{Approximate Lyapunov function $W(\cdot;\theta^*)$ (solid) and its orbital derivative $DW(\cdot;\theta^*)f$ (mesh) for Example \eqref{eq:Lyapunov_10d} on $(x_2,x_8)$-plane}
\label{fig:Lyapunov_10d_x28}
\end{figure}

\begin{figure}[htb]
\begin{center}
\includegraphics[width=6cm]{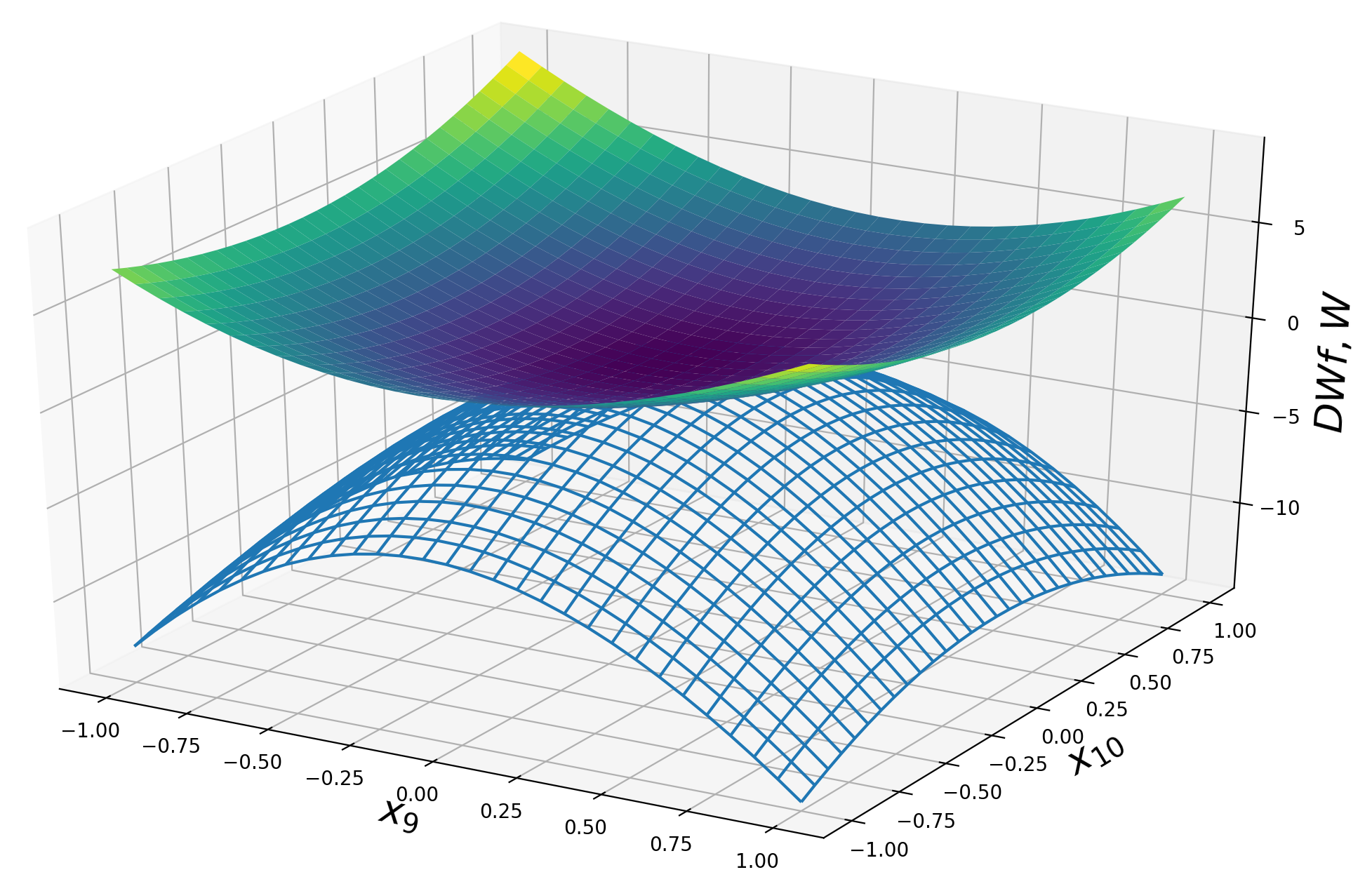}
\includegraphics[width=6cm]{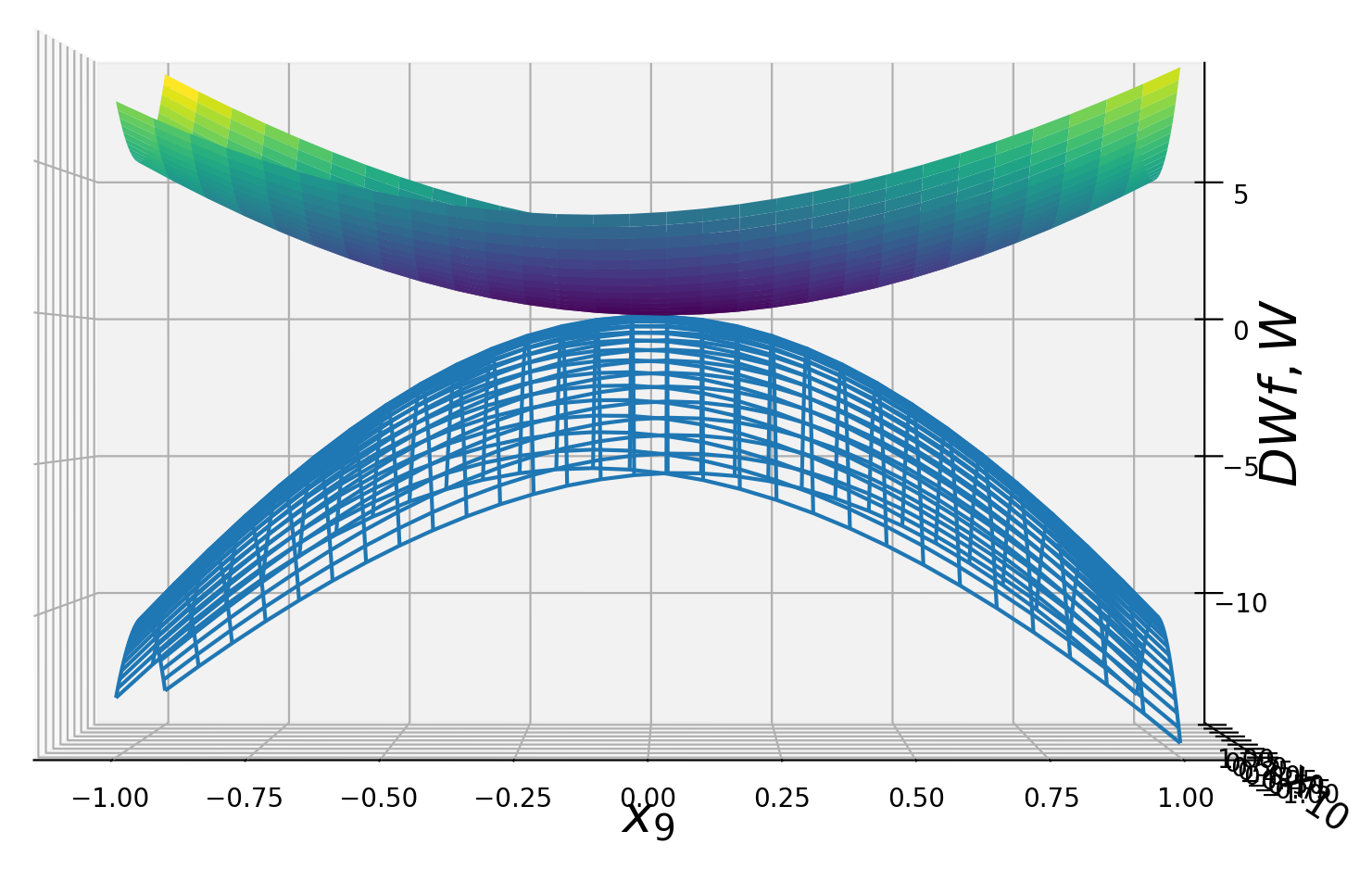}
\end{center}
\caption{Approximate Lyapunov function $W(\cdot;\theta^*)$ (solid) and its orbital derivative $DW(\cdot;\theta^*)f$ (mesh) for Example \eqref{eq:Lyapunov_10d} on $(x_9,x_{10})$-plane}
\label{fig:Lyapunov_10d_x910}
\end{figure}

\begin{figure}[htb]
\begin{center}
\includegraphics[width=4cm]{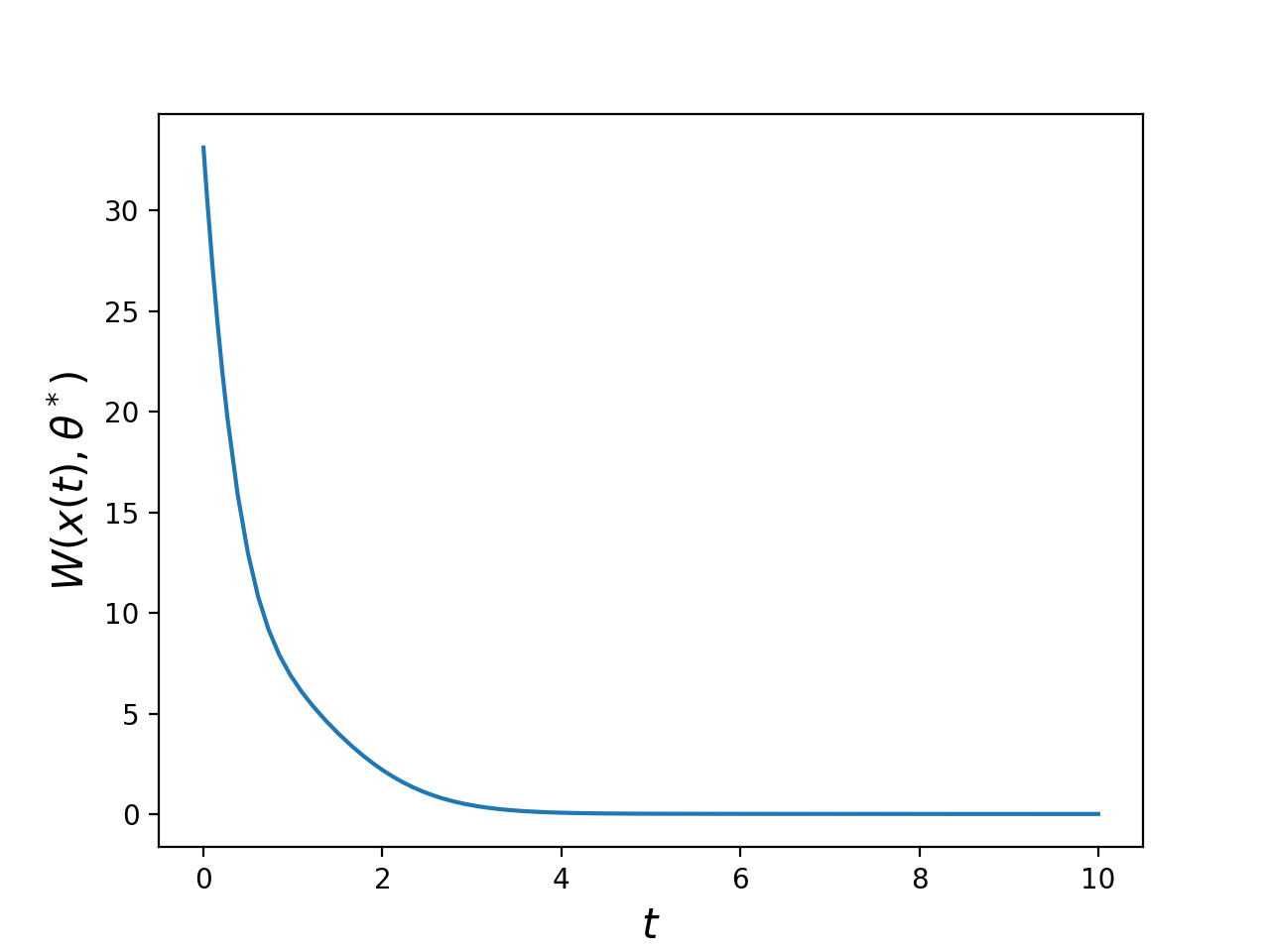}
\includegraphics[width=4cm]{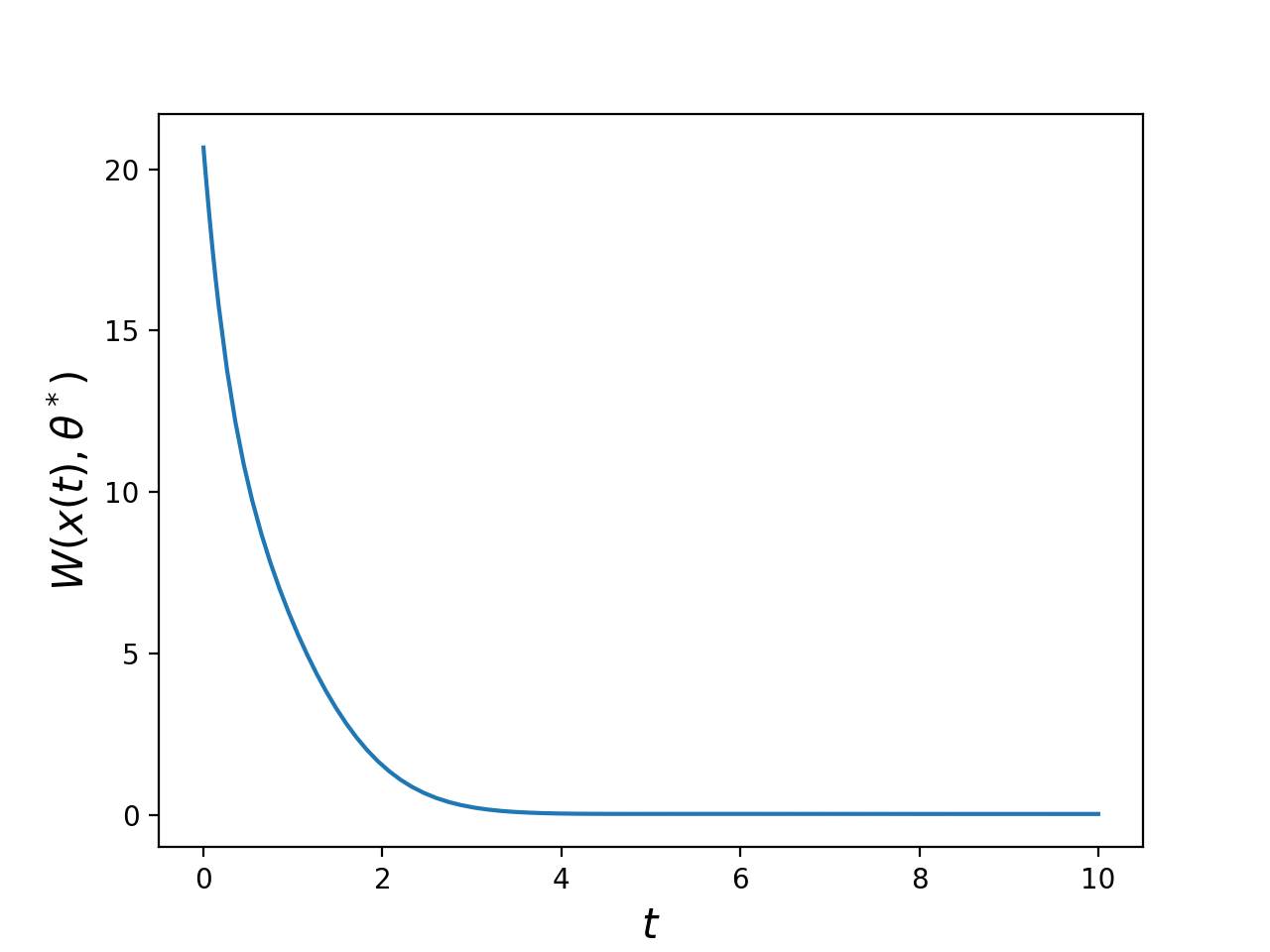}
\includegraphics[width=4cm]{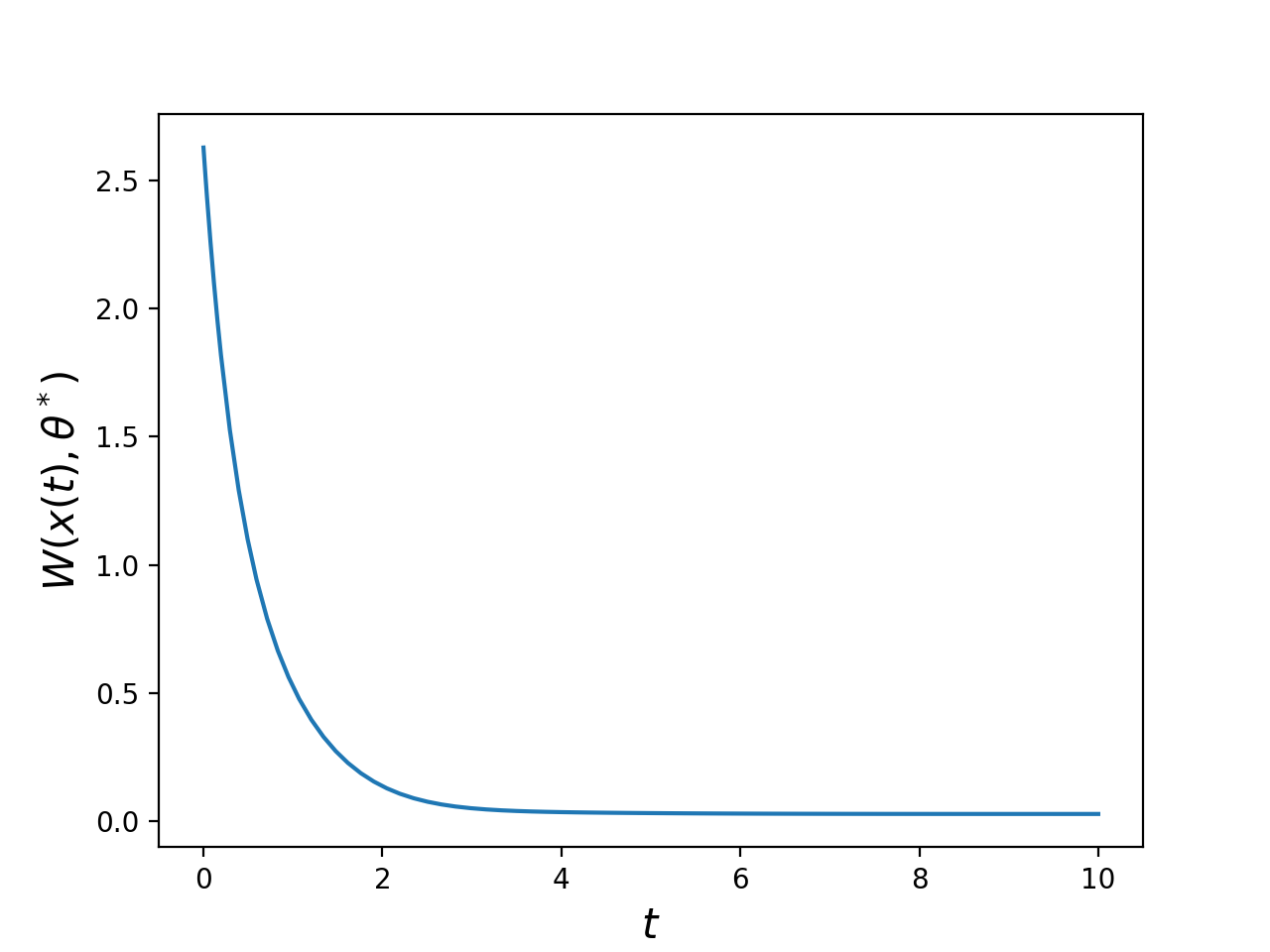}
\end{center}
\caption{Value of approximate Lyapunov function $W(x(t);\theta^*)$ along trajectories for initial values $x_0=(1,1,1,1,1,1,1,1,1,1)^T$, $(0,1,0,1,0,1,0,1,0,1)^T$, $(1,0,0,0,0,0,0,0,0,0)^T$ (left to right)}
\label{fig:Lyapunov_10d_tra}
\end{figure}

\section{Discussion}\label{sec:discussion}

In this section we discuss a few aspects and possible extensions of the results in this paper.

\begin{enumerate}[(i)]
\item From the expressions for $N$ in Proposition \ref{prop:V1} and Theorem \ref{thm:V2} one sees that for a given $\eps>0$ the storage effort only grows polynomially in the state dimension $n$, where the exponent is determined by the maximal dimension of the subsystems $d_{\max}$. The proposed approach hence avoids the curse of dimensionality, i.e., the exponential growth of the effort. There is, however, an exponential dependence on the maximal dimension $d_{\max}$ of the subsystems $\Sigma_i$ for the compositional Lyapunov functions \eqref{eq:sglf}. This is to be expected, because the construction relies on the low-dimensionality of the $\Sigma_i$ and if this is no longer given, we cannot expect the method to work efficiently.

\item We stress that our theoretical results only guarantee that the computed functions $W(\cdot;\theta^*)$ are approximations to Lyapunov functions rather than true Lyapunov functions. However, the figures of the graphs of $W$ and $DW f$ as well as further numerical tests suggest that the computed functions are indeed Lyapunov functions, except in small neighborhoods of the equilibrium $0$. However, it is currently unclear how this can be verified rigorously. In low dimensions a grid based method such as the check of \cite[inequality (3)]{HaKL14} proposed in \cite{HaKL14} might be feasible, but in higher dimensions new methods for such a verification need to be developed. Here the fact that the neural network provides an explicit analytic, albeit complex, expresssion for $W(\cdot;\theta^*)$ may be helpful.

\item There have been attempts to use small-gain theorems for grid-based constructions of Lyapunov functions, e.g., in \cite{CaGW09,Li15}. The problem of such a construction, however, is, that it computes the functions $\hat V_i$ from Theorem \ref{thm:sg} separately for the subsystems and the small-gain condition has to be checked a posteriori (which is a difficult task). The representation via the neural network does not require to check the small-gain condition nor is the precise knowledge of the subsystems necessary. 

\item \label{it:u} The reasoning in the proofs remains valid if we replace $f(x)$ by $f(x,u)$ and asymptotic stability with ISS. Indeed, we can simply incorporate $u$ as an additional external input in the small-gain formulation, which is standard in small-gain theory. Hence, the proposed network is also capable of efficiently storing ISS and iISS Lyapunov functions. Moreover, an extension to control Lyapunov functions appears attractive, as these functions allow to derive stabilizing feedback laws for nonlinear systems. However, the corresponding extension of the proposed training scheme is nontrivial and is thus subject of future research.

\item \label{it:relu} In current neural network applications ReLU activation functions $\sigma(r) = \max\{r,0\}$ are often preferred over $C^\infty$ activation functions, such as the softplus function used in our implementation (which is, in fact, a smooth approximation to the ReLU activation function). The obvious disadvantage of this concept is that the resulting function $W(x;\theta)$ is nonsmooth in $x$, which implies the need to use concepts of nonsmooth analysis for interpreting it as a Lyapunov function. While one may circumvent the need to compute the derivative of $W$ by means of using nonsmooth analysis or by passing to an integral representation of \eqref{eq:lf}
, the nonsmoothness implies that the gradient $DW$ in the training scheme needs to be replaced by an appropriate substitute.  Details are subject to future research and it remains to be explored whether the difficulties caused by the nonsmoothness of $W$ are compensated by the advantages of ReLU activation functions.

\item \label{it:maxlf} There are other types of Lyapunov function constructions based on small-gain conditions different from Definition \ref{def:sg}, e.g., a construction of the form 
\[ V(x) = \max_{i=1,\ldots,s} \rho_i^{-1}(V_i(z_i)), \]
found in \cite{DaRW10,Ruef07}. Since maximization can also be efficiently implemented in neural networks (via max pooling), such ``max-compositional'' Lyapunov functions also admit an efficient approximation via deep neural networks. However, when using this formulation we have to cope with two sources of nondifferentiability that complicate the analysis. One source is the maximization in the definition of $V$ and the other source are the functions $\rho_i^{-1}\in\KK_\infty$, which in most references are only ensured to be Lipschitz.

\item Clearly, when using the inequality-based loss function \eqref{eq:pdiloss}, then the result of the algorithm is not unique. It may thus be desirable to specify additional criteria that single out particularly useful Lyapunov functions from the set of possible solutions, such as Lyapunov functions avoiding highly degenerate level sets, guaranteeing a large domain of attraction. Such criteria have already been employed in the context of the piecewise affine approximation approach \cite{GieH14} and the ideas developed there could be also be investigated for the neural network approximation.
\end{enumerate}

\section{Conclusion}\label{sec:conclusion}
We have proposed a class of deep neural networks that allows for approximating Lyapunov functions $V$ having a compositional structure. Such Lyapunov functions exist, e.g., when the systems satisfies a small-gain condition. The number of neurons needed for an approximation with fixed accuracy depends exponentially on the maximal dimension of the subsystems in the compositional representation of $V$, but only polynomially on the overall state dimension. Thus, it provably avoids the curse of dimensionality, a feature that to the best of our knowledge is not available for similar approaches in the literature. Except for the upper bound $d_{\max}$, the network structure does not need any knowledge about the dimensions of the subsystems and the approach even allows for a subsystem structure that only becomes visible after a linear coordinate transformation. 

We also presented a loss function for a training scheme for the proposed architecture that is based on a suitable partial differential inequality and boundary conditions. By means of numerical examples we demonstrated that this approach is beneficial compared to a loss function based on a partial differential equation and that it produces excellent results in ten space dimensions. This dimension is significantly larger than those reported for other numerical approaches for nonlinear systems in the literature, particularly for grid based methods. As discussed in Section \ref{sec:discussion}, the approach allows for manifold extensions that will be subject of future research.

\providecommand{\href}[2]{#2}
\providecommand{\arxiv}[1]{\href{http://arxiv.org/abs/#1}{arXiv:#1}}
\providecommand{\url}[1]{\texttt{#1}}
\providecommand{\urlprefix}{URL }

\end{document}